%% file: main.tex
\documentclass[11pt,a4paper]{article}
\usepackage[T1]{fontenc}
\usepackage{lmodern}
\usepackage[margin=27mm]{geometry}
\usepackage{amsmath,amssymb,amsthm,mathtools}
\usepackage{microtype}
\usepackage{booktabs,tabularx,longtable,fancyvrb}
\usepackage{needspace}
\usepackage{placeins}
\usepackage{xcolor}
\usepackage{tikz}
\usetikzlibrary{arrows.meta}
\usepackage[colorlinks=true,linkcolor=blue!45!black,citecolor=blue!45!black,urlcolor=blue!45!black]{hyperref}
\usepackage[nameinlink,capitalise]{cleveref}
\newtheorem{theorem}{Theorem}
\newtheorem{corollary}[theorem]{Corollary}
\newtheorem{conjecture}{Conjecture}
\newtheorem{question}{Question}
\newtheorem{lemma}{Lemma}[section]
\newtheorem{proposition}[lemma]{Proposition}
\theoremstyle{definition}
\newtheorem*{definition}{Definition}
\theoremstyle{remark}
\newtheorem*{remark}{Remark}
\theoremstyle{plain}
\numberwithin{equation}{section}
\crefname{theorem}{Theorem}{Theorems}
\Crefname{theorem}{Theorem}{Theorems}
\crefname{corollary}{Corollary}{Corollaries}
\Crefname{corollary}{Corollary}{Corollaries}
\crefname{conjecture}{Conjecture}{Conjectures}
\Crefname{conjecture}{Conjecture}{Conjectures}
\crefname{question}{Question}{Questions}
\Crefname{question}{Question}{Questions}
\crefname{lemma}{Lemma}{Lemmas}
\Crefname{lemma}{Lemma}{Lemmas}
\crefname{proposition}{Proposition}{Propositions}
\Crefname{proposition}{Proposition}{Propositions}
\crefname{equation}{Equation}{Equations}
\Crefname{equation}{Equation}{Equations}
\crefname{section}{Section}{Sections}
\Crefname{section}{Section}{Sections}
\crefname{subsection}{Subsection}{Subsections}
\Crefname{subsection}{Subsection}{Subsections}
\crefname{appendix}{Appendix}{Appendices}
\Crefname{appendix}{Appendix}{Appendices}
\crefname{figure}{Figure}{Figures}
\Crefname{figure}{Figure}{Figures}
\crefname{table}{Table}{Tables}
\Crefname{table}{Table}{Tables}
\newcommand{\C}{C_{\mathrm{OR}}}
\newcommand{\OR}{\mathbin{\vee}}
\newcommand{\join}{\mathbin{\oplus}}
\newcommand{\F}{\mathbb F}
\newcommand{\spec}{\Delta_{\mathrm{OR}}}
\newcommand{\phif}[1]{\varphi_{#1}}
\newcommand{\rank}{\operatorname{rank}}
\newcommand{\suppref}[1]{Appendix~\ref{#1}}
\hypersetup{pdftitle={Hedetniemi's conjecture for Shannon OR-capacity: counterexamples and asymptotic separation},pdfauthor={Julius A. Zeiss},pdfsubject={Shannon OR-capacity, categorical products and asymptotic separation}}
\title{Hedetniemi's conjecture for Shannon OR-capacity:\\counterexamples and asymptotic separation}
\author{Julius A.~Zeiss\thanks{Corresponding author. Email: \href{mailto:jzeiss@physik.rwth-aachen.de}{\texttt{jzeiss@physik.rwth-aachen.de}}}\\[3pt]
{\small Institute for Quantum Information}\\
{\small RWTH Aachen University}\\
{\small Otto-Blumenthal-Stra\ss e 20, 52074 Aachen, Germany}}
\date{}

\begin{document}
\maketitle
\vspace{-1.8em}
\begin{abstract}
We disprove the analogue of Hedetniemi's conjecture for Shannon
OR-capacity, which asserts that the capacity of a categorical product
equals the smaller capacity of its factors. For every sufficiently large
complete graph, we construct a second graph of larger capacity such that
their categorical product has capacity strictly below that of the
complete graph. Further constructions show that, for every real
threshold above an explicit constant, both factor capacities can exceed
that threshold while the product capacity lies below it. We also
construct graph pairs, again with one complete factor, whose capacities
tend to infinity while their product capacity is subpolynomial in the
smaller factor capacity. Whether the product capacity must nevertheless
tend to infinity whenever both factor capacities do remains open. The
proofs use spectral interpolation, which also yields a single spectral
point violating Hedetniemi's identity throughout the complete-factor
family. Finally, we identify the categorical product as the
distinguishability graph of a random-switching model for zero-error
communication. In this model, our constructions show that switching can
reduce the zero-error rate to a vanishing fraction of the smaller
individual rate.
\end{abstract}
\par\medskip

\section{Introduction}

Can combining two graphs reduce the number of colours needed for a
proper colouring? This question lies at the heart of Hedetniemi's
conjecture~\cite{Hedetniemi} for the categorical product.
Throughout, graphs are finite, simple and nonempty. For graphs $G$ and
$H$, the categorical product $G\times H$ has vertex set $V(G)\times V(H)$,
with $(g,h)$ adjacent to $(g',h')$ if $gg'\in E(G)$ and $hh'\in E(H)$.
The classical conjecture asserted that
\[
 \chi(G\times H)=\min\{\chi(G),\chi(H)\}.
\]
El-Zahar and Sauer proved that equality holds whenever
$\min\{\chi(G),\chi(H)\}\le4$~\cite{ElZaharSauer}.
The identity also holds for several natural graph classes,
including products with a perfect factor~\cite{BurrErdosLovasz},
products of two Kneser graphs, and products of two Schrijver
graphs~\cite{HajiabolhassanMeunier}.
Shitov disproved the conjecture in general~\cite{Shitov}.
Subsequent work established counterexamples with the smaller factor
chromatic number equal to any prescribed integer at least
$5$~\cite{TardifFive}, completing the classification by this minimum.
Once equality is known to fail, it is natural to ask how small the
product chromatic number can be relative to the smaller factor
chromatic number. He and Wigderson proved an asymptotic separation by showing that the
chromatic number of the product can be smaller than the minimum factor
chromatic number by a fixed positive proportion~\cite{HeWigderson}.
Zhu improved the asymptotic ratio to at most one half~\cite{ZhuPoljakRodl}.
The refutation of the classical weak Hedetniemi conjecture is stronger
still. Raanes's construction~\cite{Raanes}, in the form discussed in
\cref{sec:weak}, keeps the product chromatic number bounded while both
factor chromatic numbers diverge.

Hedetniemi's conjecture has a natural formulation in terms of graph
homomorphisms, which also leads to analogous questions for other graph
parameters. A graph homomorphism is an adjacency-preserving
vertex map, and we write $G\to H$ when such a map exists.
In particular, a proper $n$-colouring is a homomorphism to $K_n$.
A graph parameter $p$ is homomorphism-monotone if $G\to H$ implies
$p(G)\le p(H)$. The coordinate
projections\footnote{The maps $(g,h)\mapsto g$ and $(g,h)\mapsto h$
preserve adjacency, giving graph homomorphisms $G\times H\to G$ and
$G\times H\to H$. Monotonicity of $p$ therefore yields both
$p(G\times H)\le p(G)$ and $p(G\times H)\le p(H)$.}
give $p(G\times H)\le\min\{p(G),p(H)\}$. Following
Simonyi~\cite{Simonyi}, we call the identity
\[
 p(G\times H)=\min\{p(G),p(H)\}
\]
the \emph{Hedetniemi-type equality} for $p$. The corresponding conjecture
asks whether it holds for every pair of graphs $G,H$. This identity holds for the
fractional chromatic number~\cite{Zhu} and the Lov\'asz theta number of
the complement~\cite{GodsilEtAl}, whereas it fails for the quantum
chromatic number~\cite{Zeiss}. For surveys of the conjecture and related
problems, see Zhu~\cite{ZhuSurvey1998,ZhuSurvey2025},
Sauer~\cite{SauerSurvey} and Tardif~\cite{TardifSurvey}.

In this work, we study the analogue of Hedetniemi's conjecture for
Shannon OR-capacity.
To define this parameter, we use the OR-product $G\OR H$, which has
vertex set $V(G)\times V(H)$ and requires adjacency in at least one
coordinate. This differs from the categorical product appearing in
Hedetniemi's conjecture. Write $G^{\OR r}$ for the OR-product of $r$
copies of $G$. The Shannon OR-capacity~\cite{Shannon} is
\begin{equation}\label{eq:capacity}
 \C(G)=\lim_{r\to\infty}\omega(G^{\OR r})^{1/r},
\end{equation}
where $\omega(G)\in\mathbb{N}$ denotes the number of vertices in a largest complete
subgraph of $G$ (the clique number), equivalently the largest size of a
set of pairwise adjacent vertices.
The limit exists by supermultiplicativity and Fekete's
lemma~\cite{Fekete,Shannon}. Unlike the integer-valued
chromatic number, the Shannon OR-capacity is real-valued and need not be
integral~\cite{Lovasz}.

Shannon OR-capacity originates in the problem of communicating over a
finite discrete memoryless channel without feedback, with zero
probability of decoding error~\cite{Shannon}.
When graph edges represent input symbols that cannot produce a common
output, $\omega(G^{\OR r})$ is the largest number of messages that can
be distinguished without error in $r$ channel uses. Consequently,
$\log_2\C(G)$ is the optimal zero-error rate in bits per use; see
Csonka and Simonyi~\cite[Sec.~I]{CsonkaSimonyi} for this formulation
in terms of distinguishability graphs and OR-products.

The same channel can also be described by its confusability graph
$\overline G$, the complement of $G$, in which two distinct input
symbols are adjacent when they can produce a common output.
This convention uses the strong
(AND) product. For finite graphs $G$ and $H$, the strong product
$G\boxtimes H$ has vertex set $V(G)\times V(H)$, with two distinct pairs
adjacent precisely when their entries are equal or adjacent in each
coordinate.
Let $\alpha(G)$ denote the \emph{independence number} of $G$, that
is, the largest size of a set of pairwise nonadjacent vertices.
Write $G^{\boxtimes r}$ for the strong product of $r$ copies of $G$.
The \emph{ordinary Shannon capacity}, also called
\emph{Shannon AND-capacity}, is~\cite{Shannon,Simonyi}
\[
 \Theta(G)=\lim_{r\to\infty}\alpha(G^{\boxtimes r})^{1/r}.
\]
Complementation turns cliques in $G^{\OR r}$ into independent sets in
$(\overline G)^{\boxtimes r}$, and conversely. Thus the two conventions
describe the same zero-error codes on complementary graphs, giving
$\C(G)=\Theta(\overline G)$.
Appendix~\ref{sec:products} illustrates the graph products and their
relationships.

Like the chromatic number, Shannon OR-capacity is monotone under graph
homomorphisms~\cite[Secs.~2 and~2.2]{ZhuSurvey2025}.
Simonyi asked whether it satisfies the Hedetniemi-type equality for
the categorical product~\cite{Simonyi}.

\begin{conjecture}[Shannon OR-capacity version of Hedetniemi's conjecture]
\label{conj:strong}
For every pair of graphs $G,H$,
\[
 \C(G\times H)=\min\{\C(G),\C(H)\}.
\]
\end{conjecture}

Simonyi established the conjectured equality whenever one factor admits
a homomorphism to the other, including every pair of odd cycles~\cite[Cor.~10 and
Claim~16]{Simonyi}. Equality also holds when both factors have capacity
equal to their clique number. The factors may be
chosen from perfect graphs, certain Kneser graphs,
and complements of Kneser graphs~\cite[Prop.~15]{Simonyi}.
He also proposed graph pairs based on Paley graphs as potential sources
of counterexamples to the Shannon OR-capacity conjecture~\cite[Sec.~4.3]{Simonyi}.

\Needspace{8\baselineskip}
By analogy with the classical weak Hedetniemi conjecture for graph
colouring, we formulate the following weaker asymptotic conjecture for
Shannon OR-capacity.

\begin{conjecture}[Weak Hedetniemi conjecture for Shannon OR-capacity]
\label{conj:weak}
Along any sequence of graph pairs, if both graph capacities tend to
infinity, then the capacity of their categorical product also tends
to infinity.
\end{conjecture}

Refuting this weaker conjecture would establish a much stronger failure
of the exact identity. It would give graph pairs whose capacities both
tend to infinity while their product capacity remains uniformly bounded.

Simonyi also connected \cref{conj:strong} to the asymptotic spectrum of
graphs in the OR convention.
Its parameters are homomorphism-monotone, additive under joins and
multiplicative under the OR-product. This family includes the fractional
chromatic number and the Lov\'asz theta number of the complement.
Via complementation, Zuiddam's characterization connects this family to capacity by
expressing Shannon OR-capacity as the minimum over the entire
spectrum~\cite{Zuiddam}. Simonyi used this characterization
to prove that every counterexample for Shannon OR-capacity also violates
the Hedetniemi-type equality for some parameter in the
spectrum~\cite[Thm.~8]{Simonyi}. We introduce the spectrum formally and
explain the tools used in our constructions in \cref{sec:spectral-argument}.

\subsection*{Our results}

We disprove \cref{conj:strong} even when one factor is a complete graph.
For every integer $m\ge7$, \cref{thm:finite} gives a graph of capacity
greater than $m$ whose categorical product with $K_m$ has capacity
strictly below $m=\C(K_m)$.
We first explain the mechanism through a simple example at the beginning of
\cref{sec:simple}. Our graph pairs differ from Simonyi's proposed
candidates. We also construct a counterexample with both
capacities below six (\cref{thm:sharper}(1)).
For every real threshold above an explicit constant slightly larger
than six, we construct two graphs whose capacities exceed that
threshold. Their product has capacity below it (\cref{thm:sharper}(2)).

We make progress on the asymptotic question posed by the weak
conjecture. \Cref{thm:asymptotic} shows that, with one factor still
complete, both graph capacities can tend to infinity while the product
capacity is subpolynomial in the smaller graph capacity. Whether the
product capacity can remain
uniformly bounded is still open (\cref{conj:weak}).

The interpolated spectral points used to bound the product capacities
also witness failures of the Hedetniemi-type equality within the
asymptotic spectrum.
Thus the equality can fail even for homomorphism-monotone parameters
that are additive under joins and multiplicative under the OR-product.
In fact, one fixed spectral point violates the equality with a complete
factor $K_m$ for every integer $m\ge7$ (\cref{cor:spectral}).

A further contribution is an operational formulation of the Shannon
OR-capacity version of Hedetniemi's conjecture. We show that the
categorical product is the distinguishability graph of a zero-error
communication task involving random switching between two channels
(\cref{prop:switch}). To our knowledge, this operational
interpretation of the Shannon OR-capacity version of Hedetniemi's
conjecture has not previously been made explicit.
Under this correspondence, the conjectured identity asserts that
switching preserves the smaller individual zero-error rate. Our
counterexamples therefore demonstrate that switching can impose an
additional loss, with the retained fraction of that rate tending to
zero in the asymptotic construction
(see~\cref{cor:finite-rates,cor:asym-rates}).

\paragraph{Methods.}
To obtain these counterexamples, we transfer spectral results for
ordinary Shannon capacity to the OR convention by complementation.
Zuiddam's characterization~\cite{Zuiddam} allows us to choose spectral
bounds that are sharp on individual components and to derive capacity
lower bounds from estimates valid for every spectral point. Vrana's
probabilistic refinement~\cite{Vrana} supplies the spectral interpolation
inequality~\eqref{eq:interpolate} of \cref{lem:spectral}, which combines
bounds effective on different components into a single spectral point.
This point bounds the
capacity of the entire categorical product from above. Combining this
bound with the lower bounds on the two factor capacities establishes
the capacity counterexample. The same spectral point also witnesses
failure of the Hedetniemi-type equality for that graph pair.
A complementary pair of binary orthogonality graphs
illustrates this mechanism. The sub-six refinement uses the same
interpolation principle with explicit rational matrix certificates,
compressed by graph symmetries. The asymptotic construction extends
the mechanism to many components using Alon's finite-field polynomial
method~\cite{Alon}.

\Needspace{10\baselineskip}
\paragraph{Organization of the paper.}
\Cref{sec:simple} presents the finite counterexamples: the simple
example and its spectral argument, the complete-factor family in
\cref{sec:integer-extension}, and the sharper finite thresholds in
\cref{sec:sharper}.
\Cref{sec:asymptotic} introduces the Shannon OR-capacity analogue of
the Poljak--R\"odl function, formulates the weak conjecture in these
terms, and develops the asymptotic separation.
\Cref{sec:spectral-consequences} collects the consequences for the
asymptotic spectrum. \Cref{sec:operational} gives an interpretation
in terms of zero-error communication, and \cref{sec:discussion}
discusses open questions.
Appendices~\ref{sec:notation}--\ref{sec:interpolation} provide
notation, graph-product conventions, proofs of the spectral tools and
general interpolation bounds.
Appendices~\ref{sec:finiteproof}--\ref{sec:thresholdproof} contain the
finite constructions, refinements and exact certificates.
Appendices~\ref{sec:asymproof} and~\ref{sec:channelproof} give the proofs
of the asymptotic separation and the channel interpretation, respectively.

\section{Finite counterexamples}\label{sec:simple}

We begin with a simple counterexample, proceeding in two steps. First, we construct
a graph $F_9$ and show
that $\C(F_9)\ge\sqrt{128}>9=\C(K_9)$, so the smaller factor capacity
is nine. We then outline a spectral argument giving
$\C(F_9\times K_9)\le\sqrt{72}<9$.
The categorical product therefore has smaller capacity than either
factor, contradicting \cref{conj:strong}.
We then refine the construction to obtain counterexamples at every
integer minimum $m\ge7$, yielding \cref{thm:finite}.
Finally, \cref{sec:sharper} develops further refinements for smaller
capacities and real separating thresholds.

Let $\F_2=\{0,1\}$ be the finite field with arithmetic modulo two.
Among the $2^8$ vectors in $\F_2^8$, exactly half contain an odd number
of ones. Let $U_8$ be the graph whose vertices are these 128 vectors,
with distinct $x,y$ adjacent when $x\cdot y=0$ in $\F_2$.
Its complement $\overline{U_8}$ joins exactly the distinct pairs that
$U_8$ does not join.

\Needspace{10\baselineskip}
In the communication interpretation, vertices are input symbols, and
adjacency means that two symbols can be distinguished without error
in one use. More generally, $\omega(X^{\OR r})$ is the largest number
of pairwise distinguishable words of length $r$ for a graph $X$.
Taking the $r$th root of this number and then the limit as
$r\to\infty$ gives $\C(X)$, as in~\eqref{eq:capacity}. Both $U_8$ and $\overline{U_8}$ have capacity
eight by \cref{lem:binary}.

Keep the two graphs as separate components, with no edges between
them, and write $F_9=U_8\sqcup\overline{U_8}$ for their disjoint union.
Each binary vector $x$ labels one vertex $x_{U_8}$ in the first
component and one vertex $x_{\overline{U_8}}$ in the second.
For each of the 128 labels, form the two-coordinate word
\[
 (x_{U_8},x_{\overline{U_8}}).
\]
Every word therefore uses the first component in its first coordinate
and the second component in its second coordinate.

Recall that $F_9^{\OR2}=F_9\OR F_9$ has two-coordinate words as
vertices, with two words adjacent when their corresponding entries
are adjacent in at least one coordinate. Take distinct labels $x,y$.
If $xy$ is an edge of $U_8$, the first coordinates of their words
are adjacent. Otherwise, $xy$ is an edge of $\overline{U_8}$, so
the second coordinates are adjacent. In either case, the words are
adjacent in $F_9^{\OR2}$ and can be distinguished without error.

Every pair among the 128 distinct words is therefore adjacent, so
these words form a complete subgraph $K_{128}$ (a clique of size 128)
of $F_9^{\OR2}$.
By supermultiplicativity and Fekete's lemma, the limit defining
$\C(X)$ equals $\sup_{r\ge1}\omega(X^{\OR r})^{1/r}$;
see~\eqref{eq:capacity-supremum} in Appendix~\ref{sec:spectrum}.
Thus any fixed word length gives a lower bound on capacity, and our
two-coordinate code gives
\[
 \omega(F_9^{\OR2})\ge128,
 \qquad
 \C(F_9)\ge\omega(F_9^{\OR2})^{1/2}\ge\sqrt{128}.
\]
Since $\C(K_9)=9<\sqrt{128}$, this establishes
\[
 \min\{\C(F_9),\C(K_9)\}=9.
\]
To refute \cref{conj:strong} for the two factors $F_9$ and $K_9$,
it remains to prove $\C(F_9\times K_9)<9$.
The spectral argument below uses graph parameters that bound capacity
from above: by combining two such parameters, we obtain one whose
value on both components of $F_9\times K_9$ is at most $\sqrt{72}<9$.

\subsection{The spectral argument}\label{sec:spectral-argument}

We first introduce the spectral tools needed for the product upper
bound. We then apply them to prove $\C(F_9\times K_9)<9$,
completing our simple counterexample.

For finite graphs $G$ and $H$, the \emph{join} $G\join H$ is
obtained from their disjoint union by adding all edges between
$V(G)$ and $V(H)$.

The asymptotic spectrum depends on the chosen operations and
comparison relation. Here we use graph join as addition, the
OR-product as multiplication, and graph homomorphisms for
comparison, following~\cite[Def.~6]{Simonyi}.

\Needspace{14\baselineskip}
\begin{definition}[OR asymptotic spectrum]
Let $\mathcal G$ denote the class of finite simple nonempty graphs.
The \emph{asymptotic spectrum of graphs in the OR convention},
denoted by $\spec$, consists of all graph invariants
$\varphi:\mathcal G\to\mathbb R_{\ge0}$ satisfying, for all
$G,H\in\mathcal G$,
\[
\begin{aligned}
 \varphi(K_1)&=1
 &&\text{(normalization)},\\
 \varphi(G\join H)&=\varphi(G)+\varphi(H)
 &&\text{(join additivity)},\\
 \varphi(G\OR H)&=\varphi(G)\varphi(H)
 &&\text{(OR multiplicativity)},\\
 G\to H&\implies\varphi(G)\le\varphi(H)
 &&\text{(homomorphism monotonicity)}.
\end{aligned}
\]
Its elements are called \emph{OR spectral points}.
\end{definition}

Since $K_m$ is the join of $m$ copies of $K_1$, normalization and
join additivity give $\varphi(K_m)=m$. Moreover, a clique of size
$a$ in $G^{\OR r}$ is a copy of $K_a$, whose inclusion gives a
homomorphism $K_a\to G^{\OR r}$. Thus
\[
 a=\varphi(K_a)\le\varphi(G^{\OR r})=\varphi(G)^r.
\]
Choosing $a=\omega(G^{\OR r})$, taking $r$th roots, and letting
$r\to\infty$ shows that every spectral point bounds capacity
from above: $\C(G)\le\varphi(G)$.

Zuiddam formulates the graph spectrum using disjoint union, the
strong product, and monotonicity under homomorphisms between
complements. Write $f$ for a spectral point in this convention.
The identities
\[
 \overline{G\join H}=\overline G\sqcup\overline H,
 \qquad
 \overline{G\OR H}=\overline G\boxtimes\overline H
\]
show that the assignment $\varphi(G)=f(\overline G)$ gives a
bijection between his spectrum and $\spec$. Cliques in $G^{\OR r}$
are independent sets in $(\overline G)^{\boxtimes r}$ and conversely,
so the capacity definitions give $\C(G)=\Theta(\overline G)$.
Applying Zuiddam's capacity characterization~\cite[Thm.~1.1(ii)]{Zuiddam}
to $\overline G$ gives~\eqref{eq:duality} below, including
attainment of the minimum.

The next lemma collects this characterization, the maximum rule
for disjoint unions, and a spectral interpolation inequality.
Here the maximum rule means that a fixed spectral point takes the
maximum of its values on the components of a disjoint union.
The latter two statements follow from Vrana's
results~\cite[Prop.~4.3 and Thms.~1.1--1.2]{Vrana}, translated
to the OR convention.

\Needspace{15\baselineskip}
\begin{lemma}[Spectral characterization and interpolation]\label{lem:spectral}
For every graph $X$,
\begin{equation}\label{eq:duality}
 \C(X)=\min_{\varphi\in\spec}\varphi(X).
\end{equation}
Every $\varphi\in\spec$ satisfies
\begin{equation}\label{eq:max}
 \varphi(X\sqcup Y)=\max\{\varphi(X),\varphi(Y)\}.
\end{equation}
The \emph{spectral interpolation inequality} states that if
$\varphi_1,\ldots,\varphi_k\in\spec$ and $w_j\ge0$ with
$\sum_jw_j=1$, then there is a single $\psi\in\spec$ such that
\begin{equation}\label{eq:interpolate}
 \psi(X)\le\prod_{j=1}^k\varphi_j(X)^{w_j}
 \qquad\text{for every graph }X.
\end{equation}
\end{lemma}

The proof of \cref{lem:spectral} is given in Appendix~\ref{sec:spectrum}.

We use these three statements together to bound the capacity of a
disjoint union. Equation~\eqref{eq:duality} lets us choose, for each
graph $X$, a spectral point $\varphi$ with $\varphi(X)=\C(X)$; this
choice may depend on $X$. Thus different components may have
different spectral points giving their best bounds.
The spectral interpolation inequality~\eqref{eq:interpolate} combines these points into one
spectral point $\psi$ whose values are bounded by the corresponding
weighted geometric means. Because the same $\psi$ works on every
component, the maximum rule~\eqref{eq:max} gives
\[
 \C(X\sqcup Y)\le\psi(X\sqcup Y)
 =\max\{\psi(X),\psi(Y)\}.
\]
Hence a common upper bound on $\psi(X)$ and $\psi(Y)$ also bounds
the capacity of their disjoint union.

Equation~\eqref{eq:duality} also gives lower bounds: if
$\varphi(X)\ge L$ for every $\varphi\in\spec$, then taking the
minimum gives $\C(X)\ge L$. We use this observation in
\cref{lem:joined-complement}.

\Needspace{15\baselineskip}
\par\medskip\noindent
We now return to the explicit simple counterexample introduced at
the beginning of \cref{sec:simple}, with
$F_9=U_8\sqcup\overline{U_8}$ and second factor $K_9$.
It remains to prove $\C(F_9\times K_9)<9$.
Since $\C(U_8)=\C(\overline{U_8})=8$, the attained minimum
in~\eqref{eq:duality} lets us choose spectral points
$\varphi_1,\varphi_2\in\spec$ satisfying
\[
 \varphi_1(U_8)=\varphi_2(\overline{U_8})=8.
\]
On each component product, the projections to its two factors give
the following bounds.
\begin{center}
\begin{tabular}{lcc}
\toprule
Component & Upper bound for $\varphi_1$ & Upper bound for $\varphi_2$ \\
\midrule
$U_8\times K_9$ & $8$ & $9$ \\
$\overline{U_8}\times K_9$ & $9$ & $8$ \\
\bottomrule
\end{tabular}
\end{center}
The two rows of the table interchange the bounds $8$ and $9$, so
we choose equal interpolation weights to obtain the same bound on
both components. Applying the spectral interpolation
inequality~\eqref{eq:interpolate} with $k=2$ and
$w_1=w_2=\tfrac12$ gives a single spectral point $\psi\in\spec$
such that
\[
 \psi(X)\le\varphi_1(X)^{1/2}\varphi_2(X)^{1/2}
 =\sqrt{\varphi_1(X)\varphi_2(X)}
 \qquad\text{for every graph }X.
\]
Applying this inequality to the two component products and using
the bounds in the table gives
\[
\begin{alignedat}{2}
 \psi\bigl(U_8\times K_9\bigr)
 &\le\sqrt{8\,\cdot\,9} &&=\sqrt{72},\\[4pt]
 \psi\bigl(\overline{U_8}\times K_9\bigr)
 &\le\sqrt{9\,\cdot\,8} &&=\sqrt{72}.
\end{alignedat}
\]
Both estimates hold for the same spectral point $\psi$, because
the interpolation inequality holds for every graph $X$.
The categorical product distributes over disjoint union. Since
$F_9=U_8\sqcup\overline{U_8}$, we have
\[
 F_9\times K_9
 =(U_8\times K_9)\sqcup(\overline{U_8}\times K_9).
\]
Indeed, the product vertices split according to which component of
$F_9$ contains their first coordinate. Adjacency in the categorical
product requires adjacency between the first coordinates, so no edge
can join these two groups. The maximum rule~\eqref{eq:max} therefore gives
\begin{equation}\label{eq:nine-base}
 \C(F_9\times K_9)\le\psi(F_9\times K_9)\le\sqrt{72}<9.
\end{equation}
Since nine is the smaller factor capacity, this disproves
\cref{conj:strong}.

\begin{remark}[Role of interpolation]
The same spectral point must control both components: the maximum
rule applies to each fixed spectral point, but not generally to
capacity. Indeed, $F_9$ itself illustrates the distinction. Its two
components each have capacity $8$, whereas
$\C(F_9)\ge\sqrt{128}>8$.
The same argument gives a general product bound: for arbitrary graphs
$A,B$ and every integer $m\ge1$,
\[
 \C\bigl((A\sqcup B)\times K_m\bigr)
 \le\sqrt{m\max\{\C(A),\C(B)\}}.
\]
The weighted form of this bound is proved in \cref{prop:two}
in \suppref{sec:interpolation}.
\end{remark}

\begin{remark}[Connected factors]
Adding one vertex adjacent to every vertex of $F_9$ gives a connected
graph $G=F_9\join K_1$ on 257 vertices. Using \cref{prop:nine}, \cref{prop:two} with $m=10$, and \cref{lem:connected}, we obtain
\[
 \C(G)=\sqrt{128}+1>10,
 \qquad
 \C(G\times K_{10})\le\sqrt{80}+1<10=\C(K_{10}).
\]
Thus both factors in this counterexample are connected.
Disconnectedness is therefore not essential to the failure of the
product identity.
\end{remark}

\subsection{A family of counterexamples with a complete factor}\label{sec:integer-extension}

We now refine the simple counterexample to obtain a family with complete
factor $K_m$ and smaller factor capacity exactly $m$ for every integer
$m\ge7$. The proof follows essentially the same steps as above, using
modified components and weighted spectral interpolation.

The starting case is $m=7$, the smallest integer minimum attained by
our constructions. For this example, we replace $U_8$ by the analogous
graph $U_6$ on the odd-weight vectors in $\F_2^6$. Letting $Q$ denote
the 5-regular Clebsch graph on 16 vertices~\cite{GreenwoodGleason}, we
form the first factor $X_7=U_6\sqcup(\overline{U_6}\join Q)$
by joining $Q$ to $\overline{U_6}$ and taking the disjoint union with
$U_6$. The second factor is $K_7$.
\Cref{lem:clebsch,lem:joined-complement} supply the additional capacity
estimates and the lower bound for the enlarged union. Together with
the weighted product bound in \cref{prop:two}, derived from spectral
interpolation~\eqref{eq:interpolate}, these establish the counterexample
with $K_7$. For $m>7$, let $X_m$ be obtained by joining
$K_{m-7}$ to each of the two components of $X_7$. This gives
a family $X_m$, with $2m+66$ vertices for each integer $m\ge7$,
establishing the following result.

\Needspace{8\baselineskip}
\begin{theorem}[Counterexamples with a complete factor]
\label{thm:finite}
For every integer $m\ge7$, there exists a graph $G$ such that
\[
 \C(G\times K_m)\le m-\tfrac14<m=\C(K_m)<\C(G).
\]
\end{theorem}

The full construction and proof of \cref{thm:finite} are given in
\cref{sec:seven}.

\FloatBarrier
\subsection{Sharper finite thresholds}\label{sec:sharper}

In \cref{thm:finite}, the complete factor $K_m$ has capacity $m$, while
the other factor has larger capacity. The smaller factor capacity is
therefore exactly $m$, and the construction works for every integer
$m\ge7$. We now drop the requirement that one factor be complete and show
that the Hedetniemi identity can fail even when both factor capacities
are below six. We also obtain counterexamples across a
continuous range: for every real $t\ge6.0101$, there are graphs whose
factor capacities both exceed $t$, while their product capacity is
strictly below $t$. We call such a value $t$ a \emph{separating threshold}
for the pair.

\begin{theorem}[Sharper finite thresholds]\label{thm:sharper}
\leavevmode
\begin{enumerate}
\item There exist graphs $G,H$ with $\C(G),\C(H)<6$ such that
\[
 \C(G\times H)<\sqrt{27}
 \le\min\{\C(G),\C(H)\}.
\]
\item For every real $t\ge6.0101$, there exist graphs $G,H$ such that
\[
 \C(G\times H)<t<\min\{\C(G),\C(H)\}.
\]
\end{enumerate}
\end{theorem}

Parts~(1) and~(2) are proved in Appendix~\ref{sec:subsixproof}
(\cref{subsix:thm:main}) and Appendix~\ref{sec:thresholdproof}
(\cref{thm:real-thresholds}), respectively.

\begin{remark}
By \cref{lem:integer-minimum}, \cref{thm:sharper}(2) already implies
that counterexamples exist with smaller factor capacity exactly $m$
for every integer $m\ge7$. The construction for \cref{thm:finite}
additionally provides a complete factor $K_m$, a partner on $2m+66$
vertices, and product capacity at most $m-\tfrac14$.
Allowing noncomplete factors lowers the smaller factor capacity to a
value close to $\sqrt{27}\approx5.196$, but the smallest exact integer
minimum attained by our constructions remains seven: these refinements
do not settle the cases $3,4,5,6$.
\end{remark}

The refinements in \cref{thm:sharper} strengthen the two ingredients
needed for a counterexample: lower bounds on the factor capacities and
an upper bound on the product capacity. For part~(1), explicit codes of
$27$ words in two coordinates show that both factor capacities are at
least $\sqrt{27}$. Exact matrix certificates provide the spectral
estimates that interpolation~\eqref{eq:interpolate} combines to place
the product capacity strictly below this value. The same certificates
bound both factor capacities below six, giving a counterexample at
smaller capacities than the complete-factor construction.
The graphs $G,H$ have $54$ and $27$ vertices, respectively, and the
certificates give the sharper estimate $\C(G\times H)<5.193957$.
Thus this pair separates every threshold $t\in[5.193957,\sqrt{27})$.
The graph and certificate data, together with their exact verification,
are given in \suppref{sec:subsixproof}.

For part~(2), the aim is to cover an entire range of thresholds. A single
graph pair can already separate a whole interval: bounds of the form
\[
 \C(G\times H)\le U<L
 \le\min\{\C(G),\C(H)\}
\]
show that the same pair separates every threshold $t\in(U,L)$. Its
product capacity is below $t$, while both factor capacities are above
$t$. We call $(U,L)$ a \emph{range of separating thresholds} for the pair.

We first choose the interpolation weight for an initial graph pair to
minimize the larger of the two component bounds. This gives a product
bound below $6.0101$, while the smaller factor capacity is $\sqrt{37}$.
The pair therefore separates every threshold in $[6.0101,\sqrt{37})$.
We choose several graph pairs whose ranges of separating thresholds
overlap, so that every threshold in a continuous interval is covered.
\Cref{lem:threshold-padding} extends these examples by joining a copy
of $K_j$ to each of the two components of the first factor and to the
second factor. For these constructions, a range $(U,L)$ becomes
$(U+j,L+j)$ for every integer $j\ge0$. The original covered interval
has length greater than one, so shifting it repeatedly by one leaves
no gaps in the covered thresholds. This construction therefore extends
the coverage to every threshold above the interval's lower endpoint.
Combining this coverage with the initial pair yields counterexamples
for every $t\ge6.0101$.

These constructions do not connect the uniform range to the interval
near $\sqrt{27}$. They establish strict capacity gaps for finite
examples. We next ask how small the product capacity can remain as
both factor capacities grow.

\section{Asymptotic separation and the weak conjecture}\label{sec:asymptotic}

Motivated by the classical weak Hedetniemi conjecture for graph
colouring, we formulate an analogous conjecture for Shannon OR-capacity.
It asks whether the capacity of the categorical product must tend to
infinity whenever both factor capacities do. We introduce the
corresponding analogue of the Poljak--R\"odl function and prove a
subpolynomial upper bound, leaving this conjecture open.

\subsection{Extremal functions and the weak conjecture}\label{sec:weak}

The classical weak Hedetniemi conjecture concerns the Poljak--R\"odl
function~\cite{PoljakRodl}
\[
 f_\chi(n)=\min\{\chi(G\times H):\chi(G)\ge n,\ \chi(H)\ge n\},
\]
and asks whether $f_\chi(n)\to\infty$. Raanes proposed a refutation
using products of bounded chromatic number~\cite{Raanes}.
A correction to Raanes's argument preserves his general proof strategy
and gives $f_\chi(n)\le5$ for every $n$, so the product remains
five-colourable while both factor chromatic numbers grow without
bound.\footnote{The author discovered the gap incidentally while working
on the quantum version of Hedetniemi's conjecture and communicated both
the gap and its correction directly to Raanes. The correction yields a
five-colour bound rather than the four-colour bound asserted in the
original preprint, and thus still refutes the classical weak Hedetniemi
conjecture. The author expects Raanes to revise his draft and does not
currently intend to publish the correction separately.}
This conclusion does not transfer directly to Shannon OR-capacity.
Although $\C(G)\le\chi(G)$, a large chromatic number does not by itself
establish a large capacity.

For real $t\ge1$, define the Shannon OR-capacity analogue
\begin{equation}\label{eq:extremal}
 f_C(t)=\inf\bigl\{\C(G\times H):\C(G)\ge t,\ \C(H)\ge t\bigr\}.
\end{equation}
Here $t$ is a lower bound on both factor capacities; neither capacity
is required to equal $t$. The function $f_C$ is nondecreasing, as
verified in \suppref{sec:spectrum}. In terms of this function, \cref{conj:weak} takes the form
\[
 \lim_{t\to\infty}f_C(t)=\infty.
\]
Equivalently, for every $B>0$ there is a $T>0$ such that, for all
graphs $G,H$,
\begin{equation}\label{eq:weak-uniform}
 \min\{\C(G),\C(H)\}\ge T
 \quad\Longrightarrow\quad \C(G\times H)>B.
\end{equation}
The equivalence of these formulations is justified in
\suppref{sec:spectrum}.

\subsection{Asymptotic separation}\label{sec:asymptotic-bound}

Although we do not refute the weak Hedetniemi conjecture for Shannon
OR-capacity, we make progress in this direction by proving an asymptotic
separation: both factor capacities can tend to infinity while the product
capacity is subpolynomial in the smaller factor capacity. This holds even
when both factors are connected and one is complete.

\begin{theorem}[Asymptotic separation]\label{thm:asymptotic}
For every sufficiently large integer $m$, there is a connected graph
$G_m$ satisfying
\begin{equation}\label{eq:asym-factor}
 \C(G_m)\ge m+1,
 \qquad \C(G_m\times K_m)\le
 \exp\bigl((\log m)^{1/2+o(1)}\bigr).
\end{equation}
The smaller factor capacity is therefore exactly $m$. More precisely,
there are absolute constants $K>0$ and $m_0$ such that these graphs can
be chosen, for every integer $m\ge m_0$, to satisfy
\begin{equation}\label{eq:asym-product}
 \C(G_m\times K_m)\le
 \exp\!\left(K\sqrt{\log m}\,
       \frac{(\log\log m)^2}{\log\log\log m}\right).
\end{equation}
Consequently, for every $\delta>0$ and all sufficiently large real $t$,
\begin{equation}\label{eq:clean-asym}
 f_C(t)\le \exp\!\bigl((\log t)^{1/2+\delta}\bigr).
\end{equation}
\end{theorem}

The full construction, parameter choices, and proof of
\cref{thm:asymptotic} are given in Appendix~\ref{sec:asymproof}.

In particular, $f_C(t)=t^{o(1)}$, so
\begin{equation}\label{eq:limits}
 \lim_{t\to\infty}\frac{f_C(t)}{t}=0,
 \qquad
 \lim_{t\to\infty}\frac{\log f_C(t)}{\log t}=0.
\end{equation}
Consequently, any lower bound for $f_C(t)$ that tends to infinity
must grow more slowly than every fixed positive power of $t$.

The idea behind the proof of \cref{thm:asymptotic} is to construct graphs with large capacities whose categorical
product has much smaller capacity. One factor is again the complete
graph $K_m$, with $\C(K_m)=m$. We now construct the other factor with
capacity at least $m$, so that the smaller factor capacity is exactly $m$.

The construction replaces the complementary pair $U_8,\overline{U_8}$
from \cref{sec:simple} with a family of graphs on a common set of $n$
vertex labels, chosen so that every pair of distinct labels is adjacent
in at least one graph; equivalently, their edge sets together cover all
edges of $K_n$. \Cref{lem:partition} provides such a family of $N$ graphs.
Let $F$ be their disjoint union. Write $x_i$ for the copy of vertex
label $x$ in the $i$th graph of the
disjoint union $F$. For each label $x$, form the length-$N$ word
$(x_1,\ldots,x_N)$, using one vertex from each graph in this fixed order.

For distinct labels $x,y$, this covering condition gives a graph in which
they are adjacent. If this is the $i$th graph, then $x_i$ and $y_i$ are
adjacent in $F$, so the corresponding words are adjacent in their
$i$th coordinate and hence in $F^{\OR N}$. Since this holds for every
pair of distinct labels, the $n$ words form a clique of size $n$.
By~\eqref{eq:capacity-supremum} in Appendix~\ref{sec:spectrum},
\[
 \C(F)\ge\omega(F^{\OR N})^{1/N}\ge n^{1/N}.
\]
We choose the construction parameters as functions of $m$ so that
$n^{1/N}\ge m$. Thus $\C(F)\ge m$, and since $\C(K_m)=m$, the smaller
factor capacity is exactly $m$:
\[
 \min\{\C(F),\C(K_m)\}=m.
\]

To bound the product capacity, we use spectral points associated with
$k$ finite fields. For each graph in the family, \cref{lem:partition}
supplies matrices of rank at most $D$ over $k-h$ of these fields.
These matrices have diagonal entries $1$ and vanish on the edges of
that graph, so \cref{lem:haemers} bounds the corresponding spectral
values by $D$. We call the remaining $h$ fields exceptional; the bound
$D$ is not required there, and these fields may depend on the graph.
For the product of each graph with $K_m$, projection to that graph
retains the bound $D$ in its nonexceptional fields, while projection to
$K_m$ gives the bound $m$ in its exceptional fields. Equal-weight
spectral interpolation~\eqref{eq:interpolate} produces a single spectral
point controlling all these products. Since $F\times K_m$ is their
disjoint union, applying the maximum rule to this spectral point gives,
as formalized in \cref{lem:many},
\[
 \C(F\times K_m)\le D^{1-h/k}m^{h/k}.
\]

The gain comes from making the fraction $h/k$ of exceptional fields
tend to zero. The same parameter choices that ensure $n^{1/N}\ge m$,
and hence $\C(F)\ge m$, also give $D=m^{o(1)}$ and $h/k=o(1)$.
Thus the smaller factor capacity is exactly $m$, whereas the product
capacity is subpolynomial in $m$.

\begin{remark}[Connected factors]
To obtain the connected factor in \cref{thm:asymptotic}, add one
universal vertex to $F$ and set $G_m=F\join K_1$. By
\cref{lem:connected},
\[
 \C(G_m)=\C(F)+1\ge m+1,
 \qquad
 \C(G_m\times K_m)\le\C(F\times K_m)+1.
\]
Thus the smaller factor capacity remains exactly $m$, while the
product capacity remains subpolynomial in $m$.
\end{remark}

The quantitative choices in Appendix~\ref{sec:asymproof} yield the
bound in \cref{thm:asymptotic}.

The upper bound in \cref{thm:asymptotic} tends to infinity and therefore
does not resolve \cref{conj:weak}. A refutation would
require an absolute constant $B$ and graph pairs $(G_n,H_n)$ with
\[
 \min\{\C(G_n),\C(H_n)\}\longrightarrow\infty,
 \qquad \C(G_n\times H_n)\le B\quad\text{for every }n.
\]
Conversely, proving the conjecture would require a lower bound tending
to infinity for $f_C$, which our argument does not provide.
The equivalence between failure of \cref{conj:weak} and the existence
of such a sequence is proved in \suppref{sec:spectrum}.

\section{Consequences for the asymptotic spectrum}\label{sec:spectral-consequences}

Our counterexamples show that even a graph parameter satisfying all the
defining properties of a spectral point can assign a categorical product
a smaller value than either factor. This follows from Simonyi's theorem:
every failure of Hedetniemi's identity for Shannon OR-capacity yields a
failure for some point of the asymptotic spectrum $\spec$ introduced in
\cref{sec:spectral-argument}~\cite[Thm.~8]{Simonyi}; see also Zhu's
survey~\cite[Sec.~2.2]{ZhuSurvey2025}. In general, the witnessing point
may depend on the graph pair. For the complete-factor family in
\cref{sec:integer-extension}, we obtain a stronger uniform conclusion:
spectral interpolation~\eqref{eq:interpolate} supplies one fixed spectral
point that witnesses failure for every member.

\Needspace{8\baselineskip}
\begin{corollary}[Failure of Hedetniemi's identity in the asymptotic spectrum]\label{cor:spectral}
There exists a spectral point $\psi\in\spec$ such that, for every
integer $m\ge7$, there is a graph $G_m$ satisfying
\[
 \psi(G_m\times K_m)<m
 =\min\{\psi(G_m),\psi(K_m)\}.
\]
\end{corollary}

The proof follows from the fixed spectral point constructed for
\cref{thm:finite} and is given in \cref{sec:seven}.

\begin{remark}[Sharper finite examples]
The constructions in \cref{thm:sharper} also give spectral failures
at smaller values and at real separating thresholds. For the sub-six
pair, a witnessing spectral point has product value below $\sqrt{27}$
and both factor values in $[\sqrt{27},6)$, as shown in
Appendix~\ref{sec:subsixproof}. For every real $t\ge6.0101$, the
construction in Appendix~\ref{sec:thresholdproof} gives graphs $G_t,H_t$
and a spectral point $\psi_t$ with
$\psi_t(G_t\times H_t)<t<\min\{\psi_t(G_t),\psi_t(H_t)\}$.
Here $\psi_t$ may depend on $t$.
\end{remark}

Thus homomorphism monotonicity, additivity under joins, and
multiplicativity under the OR-product do not imply Hedetniemi's
identity for the categorical product, even
when one factor is complete. This contrasts with the fractional
chromatic number and the Lov\'asz theta number of the complement,
two spectral points for which the identity holds~\cite{Zhu,GodsilEtAl}.

\section{Operational interpretation}\label{sec:operational}

The categorical product has a natural interpretation in zero-error
communication: a sender supplies inputs to two channels, and a random
switch determines which output reaches the receiver. Under this
interpretation, our counterexamples show that switching can reduce the
communication rate below what either channel supports individually.

For a finite discrete memoryless channel $W$, let $D(W)$ be its
\emph{distinguishability graph}: vertices are inputs, and two inputs are
adjacent when their possible output sets are disjoint. Write $M_0(W,n)$
for the largest number of messages that can be transmitted with zero
error in $n$ uses without feedback. Shannon's coding
relations~\cite{Shannon}, expressed in the distinguishability-graph
convention of Csonka and Simonyi~\cite[Sec.~I]{CsonkaSimonyi}, give
\begin{equation}\label{eq:channel-rate}
\begin{aligned}
 M_0(W,n)&=\omega\bigl(D(W)^{\OR n}\bigr),\\
 R_0(W)&=\sup_{n\ge1}\frac{\log_2 M_0(W,n)}{n}
        =\log_2\C(D(W)).
\end{aligned}
\end{equation}
Here $R_0(W)$ is measured in bits per use. Every finite graph is the
distinguishability graph of some channel; see
Alon~\cite[Sec.~2]{AlonGraphPowers} for the equivalent statement for
confusability graphs. An explicit realization is given in
\suppref{sec:channelproof}.

Let $\mathcal A$ and $\mathcal B$ be finite input alphabets, and let
$\mathcal Y$ and $\mathcal Z$ be finite output alphabets.
Consider channels $W_G:\mathcal A\to\mathcal Y$ and
$W_H:\mathcal B\to\mathcal Z$ with distinguishability graphs $G$ and $H$.
Before transmission, the sender fixes for each message a sequence of
$n$ input pairs $(a_i,b_i)\in\mathcal A\times\mathcal B$, where
$i=1,\ldots,n$ labels the channel use. At use $i$, the inputs $a_i$
and $b_i$ are supplied to $W_G$ and $W_H$, respectively.
No information about which channel was selected or what output it
produced is returned to the sender during transmission.
For each message, the sender may choose any sequence of input pairs
from $(\mathcal A\times\mathcal B)^n$.
At each use, independently of the message, channel randomness and
previous uses, a switch selects the first channel with probability
$p\in(0,1)$ and the second otherwise. The receiver observes the selected
channel and its output. The resulting memoryless channel
$W_{\mathrm{sw}}$ has input alphabet $\mathcal A\times\mathcal B$ and
output alphabet $(\{0\}\times\mathcal Y)\cup(\{1\}\times\mathcal Z)$.
An output $(s,u)$ records the selected channel ($s=0$ for $W_G$ and
$s=1$ for $W_H$) and its output $u$. For each input
$(a,b)\in\mathcal A\times\mathcal B$, the probability of receiving
$(s,u)$ is
\[
 W_{\mathrm{sw}}\bigl((s,u)\mid(a,b)\bigr)
 =\begin{cases}
   p\,W_G(u\mid a), & s=0,\quad u\in\mathcal Y,\\[2pt]
   (1-p)\,W_H(u\mid b), & s=1,\quad u\in\mathcal Z.
  \end{cases}
\]
Zero-error decoding must succeed for every message and every possible
switch and output sequence. This is a state-dependent channel with
state information available only to the
receiver~\cite[Sec.~7.4]{ElGamalKim}. For message $j$, write
$(a_i(j),b_i(j))$ for the input pair at use $i$, and $(S_i,U_i)$ for
the corresponding random labelled output. \Cref{fig:switching}
illustrates the task.

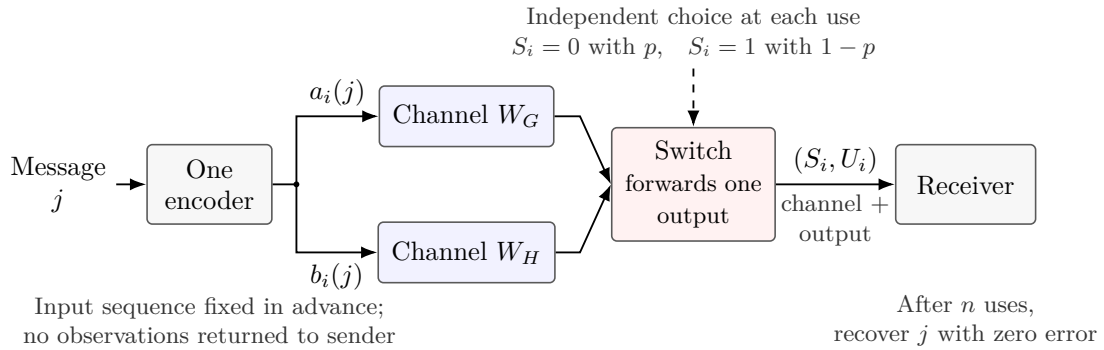
\begin{figure}[!ht]
\centering
\begin{tikzpicture}[
 x=1cm,y=1cm,>=Latex,
 every node/.style={font=\small,align=center},
 block/.style={draw=black!65,rounded corners=2pt,fill=black!3,
               line width=.6pt,minimum height=1.05cm,inner sep=5pt},
 channel/.style={block,fill=blue!5,minimum width=1.55cm,
                 minimum height=.8cm},
 flow/.style={->,line width=.7pt},
 note/.style={font=\footnotesize,text=black!80}
]
\node (message) at (0,0) {Message\\$j$};
\node[block,minimum width=1.65cm] (encoder) at (2,0) {One\\encoder};
\node[channel] (g) at (5.4,.9) {Channel $W_G$};
\node[channel] (h) at (5.4,-.9) {Channel $W_H$};
\node[block,fill=red!5,minimum width=1.65cm] (switch) at (8.4,0)
 {Switch\\\footnotesize forwards one\\\footnotesize output};
\node[block,minimum width=1.85cm] (receiver) at (12,0) {Receiver};
\node[note] (state) at (8.4,2)
 {Independent choice at each use\\
  $S_i=0$ with $p$,\quad $S_i=1$ with $1-p$};
\draw[flow] (message.east) -- (encoder.west);
\draw[line width=.7pt] (encoder.east) -- (3.15,0);
\fill (3.15,0) circle (1.1pt);
\draw[flow] (3.15,0) |- node[pos=.75,above] {$a_i(j)$} (g.west);
\draw[flow] (3.15,0) |- node[pos=.75,below] {$b_i(j)$} (h.west);
\draw[flow] (g.east) -- (6.9,.9) -- (switch.west);
\draw[flow] (h.east) -- (6.9,-.9) -- (switch.west);
\draw[flow,dashed] (state.south) -- (switch.north);
\draw[flow] (switch.east) -- node[above] {$(S_i,U_i)$}
 node[below,note] {channel +\\output} (receiver.west);
\node[note] at (2,-1.8)
 {Input sequence fixed in advance;\\no observations returned to sender};
\node[note] at (12,-1.8)
 {After $n$ uses,\\recover $j$ with zero error};
\end{tikzpicture}
\caption{Random switching between two channels. The receiver observes the
selected channel and its output, and decoding must succeed for every
possible switch and output sequence. OpenAI Codex assisted with
preparation of the TikZ source.}
\label{fig:switching}
\end{figure}

\Needspace{9\baselineskip}
\begin{proposition}[Random switching realizes the categorical product]
\label{prop:switch}
For every $p\in(0,1)$,
\begin{equation}\label{eq:switch-graph}
 D(W_{\mathrm{sw}})=G\times H.
\end{equation}
Consequently, for every $n\ge1$,
\[
 M_0(W_{\mathrm{sw}},n)=\omega\bigl((G\times H)^{\OR n}\bigr),
 \qquad R_0(W_{\mathrm{sw}})=\log_2\C(G\times H).
\]
Both quantities are independent of $p$.
\end{proposition}

Indeed, two paired inputs are distinguishable precisely when their
component inputs are distinguishable in both channels. This
gives the categorical product for a single use, whereas the OR-power
governs repeated uses. The full proof is given in
\suppref{sec:channelproof}.

Since every finite graph has a channel realization, the Shannon
OR-capacity version of Hedetniemi's conjecture is therefore equivalent
to the assertion that
\begin{equation}\label{eq:switch-benchmark}
 R_0(W_{\mathrm{sw}})=\min\{R_0(W_G),R_0(W_H)\}
\end{equation}
for every pair of channels. Our counterexamples show that this equality
can fail even when one channel is noiseless.

Such a loss does not occur for ordinary vanishing-error capacity,
denoted by $C_{\mathrm{Sh}}$. Indeed, the same switching model satisfies
\[
 C_{\mathrm{Sh}}(W_{\mathrm{sw}})
 =pC_{\mathrm{Sh}}(W_G)+(1-p)C_{\mathrm{Sh}}(W_H),
\]
which is at least the smaller component capacity.
\suppref{sec:channelproof} proves this identity. Zero-error communication
must also accommodate arbitrarily unlikely switch sequences, which
accounts for both the additional loss and its independence from $p$.

Return to the simple pair $F_9,K_9$ from \cref{sec:simple}.
Let $W_F$ realize $F_9$, and let $W_9$ be a noiseless nine-symbol
channel. For every $p\in(0,1)$,
\begin{equation}\label{eq:finite-rates}
\begin{aligned}
 R_0(W_F)&=\tfrac12\log_2 128>\log_2 9=R_0(W_9),\\
 R_0(W_{\mathrm{sw}})&\le\tfrac12\log_2 72<\log_2 9.
\end{aligned}
\end{equation}
The loss is already visible in two uses:
\begin{equation}\label{eq:finite-two-use}
 M_0(W_F,2)=128,\qquad M_0(W_9,2)=81,\qquad
 M_0(W_{\mathrm{sw}},2)\le72.
\end{equation}
Thus either channel individually can transmit 81 messages without
error in two uses, whereas their switched combination can transmit at
most 72. The bounds follow from the explicit two-coordinate code and
the capacity estimates for $F_9$ and $F_9\times K_9$; the details are
given in \suppref{sec:channelproof}.

\Cref{thm:finite,thm:sharper} give the following broader coding
consequences.

\Needspace{10\baselineskip}
\begin{corollary}[Finite loss under switching]\label{cor:finite-rates}
In each part, the channels are chosen independently of $p$, and the
conclusions hold for every $p\in(0,1)$.
\begin{enumerate}
\item For every integer $m\ge7$, there are channels $W_G,W_m$, with
$W_m$ noiseless on $m$ symbols, such that
\[
 R_0(W_{\mathrm{sw}})\le\log_2(m-\tfrac14)
 <\log_2m=R_0(W_m)<R_0(W_G).
\]
\item There are channels $W_G,W_H$ with both individual zero-error
rates below $\log_2 6$ such that
\[
 R_0(W_{\mathrm{sw}})<\tfrac12\log_2 27
 \le\min\{R_0(W_G),R_0(W_H)\}.
\]
Each channel admits a zero-error code of 27 messages in two uses,
whereas the switched channel admits at most 26.
\item For every real target rate $r\ge\log_2(6.0101)$, there are
channels $W_G,W_H$ such that
\[
 R_0(W_{\mathrm{sw}})<r<\min\{R_0(W_G),R_0(W_H)\}.
\]
\end{enumerate}
\end{corollary}

The proof is given in \suppref{sec:channelproof}.
These finite examples establish strict loss and locate it across a
range of rates. The asymptotic construction shows how severe the
relative loss can become along a growing family of channels.
Here $m$ indexes a family of channels. For each fixed $m$, the
zero-error capacities already allow codes of arbitrarily large block
length; the limit $m\to\infty$ concerns the channel family.

\Needspace{10\baselineskip}
\begin{corollary}[Asymptotic loss under switching]\label{cor:asym-rates}
For every sufficiently large integer $m$, there are channels $W_{G_m}$
and $W_m$, with $W_m$ noiseless on $m$ symbols, whose smaller individual
zero-error rate is
\[
 R_m=\min\{R_0(W_{G_m}),R_0(W_m)\}=\log_2m.
\]
For every $p\in(0,1)$, their switched channel satisfies
\begin{equation}\label{eq:asym-rates}
 R_0(W_{\mathrm{sw},m})\le R_m^{1/2+o(1)},
 \qquad \frac{R_0(W_{\mathrm{sw},m})}{R_m}\longrightarrow0
 \quad(m\to\infty).
\end{equation}
\end{corollary}

The proof is given in \suppref{sec:channelproof}.

Switching can severely reduce the zero-error communication rate
compared with using either channel alone. In our construction, both
individual rates grow without bound, yet switching reduces the
achievable rate to an arbitrarily small fraction of the smaller
individual rate. This does not establish whether the switched rate
itself remains bounded.

\FloatBarrier
\section{Discussion}\label{sec:discussion}

We have shown that the Shannon OR-capacity of a categorical product can
be strictly smaller than the capacity of either factor. This disproves
the Shannon OR-capacity version of Hedetniemi's conjecture. For every
integer $m\ge7$, we construct a counterexample in which one factor is
the complete graph $K_m$ and the other has capacity greater than $m$.
We also give a counterexample in which both factor capacities are below
six.

Our constructions also show that the product identity can fail for
parameters in the asymptotic spectrum. One fixed spectral point
witnesses this failure throughout the complete-factor family. In the
operational interpretation, our results mean that random switching
between two channels can reduce the zero-error communication rate below
both individual rates. This loss already appears in the number of
messages that can be transmitted without error in just two uses. Along
our asymptotic family, both individual rates grow without bound while
the fraction of the smaller rate retained under switching tends to zero.

Several questions remain at smaller capacities. The product identity
holds whenever the smaller factor capacity is $1$ or $2$
\cite[Rem.~5]{Simonyi}. Whether it holds when the smaller factor capacity
is exactly $3$, $4$, $5$, or $6$ remains open.

\Needspace{9\baselineskip}
\begin{samepage}
The first unresolved integer case is therefore the following.
\begin{question}[Capacity three]\label{q:capacity-three}
Do all graphs $G,H$ satisfy
\[
 \min\{\C(G),\C(H)\}=3
 \quad\Longrightarrow\quad
 \C(G\times H)=3?
\]
\end{question}
\end{samepage}

Thus, the smallest integer value of the smaller factor capacity
admitting a counterexample lies between $3$ and $7$. Determining this
value, and whether counterexamples exist at every larger integer value,
remains open.

Since Shannon OR-capacity need not be an integer, there is also a
separate question about real thresholds: for which real numbers $t$ can
both factor capacities exceed $t$ while the product capacity is below
$t$? Our constructions establish this for a small interval of thresholds
immediately below $\sqrt{27}$, and for every threshold above an explicit
constant slightly larger than $6$.
We do not yet have a complete description of these thresholds or a
continuous range connecting the two ranges above. These results do not
require either factor capacity to equal $t$. In particular, the example
with both capacities below six does not resolve the exact integer cases
$3,4,5,6$.

At large capacities, we prove that both factor capacities can tend to
infinity while the product capacity remains below every fixed positive
power of the smaller factor capacity, once that capacity is sufficiently
large. This is possible even with connected factors and one complete
factor.

The weak Hedetniemi conjecture for Shannon OR-capacity nevertheless
remains open. It asks whether the product capacity must tend to infinity
whenever both factor capacities do. To disprove it, one would need graph
pairs whose factor capacities tend to infinity while their product
capacities stay below one fixed constant. Our asymptotic bound does not
establish this, and we also do not prove that the product capacity must
diverge.

\section*{Acknowledgements}

The author acknowledges support from the European Research Council
(ERC Grant Agreement No.\ 948139) and from the Deutsche
Forschungsgemeinschaft (DFG, German Research Foundation) under Germany's
Excellence Strategy -- EXC 2004/2 -- 390534769 (ML4Q-2).
The funders had no role in the research, preparation of the manuscript,
or decision to submit the article for publication.

\section*{Declaration of competing interest}

The author declares no competing interests.

\section*{Data and code availability}

The graph definitions, finite codes and exact certificates are included
in the supplementary appendices. The accompanying public
repository~\cite{ZeissData} provides the verification programs,
machine-readable graph and certificate data, and instructions for
reproducing the finite checks. These programs use only the Python
standard library. No numerical solver is needed to verify the
certificates.

\section*{Declaration of generative AI and AI-assisted technologies}

The author used OpenAI Codex (OpenAI) to assist with mathematical
exploration, developing and reviewing arguments, literature searches,
preparation of the manuscript and TikZ figures, and development and
review of verification code. The verification subsection of
\suppref{sec:subsixproof} describes
the computational checks and the role of AI assistance. The author
takes full responsibility for the content, accuracy, originality and
integrity of the manuscript, including all AI-assisted material.

\clearpage
\appendix
\crefname{theorem}{Thm.}{Thms.}
\Crefname{theorem}{Thm.}{Thms.}
\crefname{corollary}{Cor.}{Cors.}
\Crefname{corollary}{Cor.}{Cors.}
\crefname{conjecture}{Conj.}{Conjs.}
\Crefname{conjecture}{Conj.}{Conjs.}
\crefname{question}{Q.}{Qs.}
\Crefname{question}{Q.}{Qs.}
\crefname{lemma}{Lem.}{Lems.}
\Crefname{lemma}{Lem.}{Lems.}
\crefname{proposition}{Prop.}{Props.}
\Crefname{proposition}{Prop.}{Props.}
\crefname{equation}{Eq.}{Eqs.}
\Crefname{equation}{Eq.}{Eqs.}
\crefname{section}{Sec.}{Secs.}
\Crefname{section}{Sec.}{Secs.}
\crefname{subsection}{Subsec.}{Subsecs.}
\Crefname{subsection}{Subsec.}{Subsecs.}
\crefname{appendix}{App.}{Apps.}
\Crefname{appendix}{App.}{Apps.}
\crefname{figure}{Fig.}{Figs.}
\Crefname{figure}{Fig.}{Figs.}
\crefname{table}{Tab.}{Tabs.}
\Crefname{table}{Tab.}{Tabs.}
\renewcommand{\suppref}[1]{Supplementary App.~\ref{#1}}
\renewcommand{\thesection}{\Alph{section}}
\numberwithin{theorem}{section}
\numberwithin{figure}{section}
\numberwithin{table}{section}
\renewcommand{\theHsection}{supp.\Alph{section}}
\renewcommand{\theHequation}{supp.\Alph{section}.\arabic{equation}}
\renewcommand{\theHtheorem}{supp.\Alph{section}.\arabic{theorem}}
\renewcommand{\theHfigure}{supp.\Alph{section}.\arabic{figure}}
\renewcommand{\theHtable}{supp.\Alph{section}.\arabic{table}}
\begin{center}
 {\Large\bfseries Supplementary appendices}\\[4pt]
 {\normalsize Conventions, proofs and channel interpretation}
\end{center}

\begingroup
\small
\noindent\textbf{Contents}\par\smallskip
\newcommand{\suppcontentsline}[1]{%
 \noindent\hyperref[#1]{\makebox[2.2em][l]{\ref*{#1}}\nameref*{#1}}%
 \nobreak\dotfill\hyperref[#1]{\pageref*{#1}}\par}
\suppcontentsline{sec:notation}
\suppcontentsline{sec:products}
\suppcontentsline{sec:spectrum}
\suppcontentsline{sec:interpolation}
\suppcontentsline{sec:finiteproof}
\suppcontentsline{sec:subsixproof}
\suppcontentsline{sec:thresholdproof}
\suppcontentsline{sec:asymproof}
\suppcontentsline{sec:channelproof}
\endgroup
\par\medskip

\section{Notation and conventions}\label{sec:notation}

This appendix collects the notation and conventions used throughout
the paper. It provides a reference for the graph operations, capacity
parameters and asymptotic estimates in the main text.

All graphs are finite, simple and nonempty. A graph homomorphism is a
map $f:V(G)\to V(H)$ such that $xy\in E(G)$ implies
$f(x)f(y)\in E(H)$; it need not be injective. Graphs are
\emph{homomorphically equivalent} if homomorphisms exist in both
directions. A real-valued graph invariant $p$ is
\emph{homomorphism-monotone} if $G\to H$ implies $p(G)\le p(H)$.
Such invariants agree on homomorphically equivalent graphs; examples
include $\omega$, $\chi$ and $\C$.

Logarithms in asymptotic estimates are natural; communication rates
use $\log_2$ and are measured in bits per channel use. For a real-valued
function $a$ and a positive function $b$, the notation $a=O(b)$ means
that $|a|/b$ remains bounded,
$a=o(b)$ means $a/b\to0$, and $a\sim b$ means $a/b\to1$, as the
indicated parameter tends to infinity. In particular, for $a(t)\ge1$,
$a(t)=t^{o(1)}$ means $\log a(t)/\log t\to0$ and describes growth slower than
every fixed positive power of $t$.

\par\addvspace{.8\baselineskip}
\begingroup
\small
\renewcommand{\arraystretch}{1}
\noindent
\begin{tabularx}{\textwidth}{@{}p{.29\textwidth}X@{}}
\toprule
\textbf{Notation} & \textbf{Meaning}\\
\midrule
$G\to H$, $\overline G$
 & A homomorphism from $G$ to $H$ exists; the complement of $G$.\\
$K_m$, $I_m$, $P_3$
 & Complete and edgeless graphs on $m$ vertices; the three-vertex path.\\
$G\times H$
 & Categorical product, with adjacency in both coordinates.\\
$G\boxtimes H$
 & Strong (AND) product, in which distinct pairs are adjacent when their entries are equal or adjacent in each coordinate.\\
$G\OR H$
 & OR-product, with adjacency in at least one coordinate.\\
$G^{\OR r}$, $G^{\boxtimes r}$
 & The $r$-fold OR and strong powers, respectively.\\
$A\sqcup B$, $A\join B$
 & Disjoint union; join, obtained by adding all edges between the factors.\\
$\omega(G)$, $\alpha(G)$, $\chi(G)$
 & Clique number, independence number and chromatic number.\\
$\chi_f(G)$
 & Fractional chromatic number, defined in \suppref{sec:spectrum}.\\
$\vartheta(G)$, $\overline\vartheta(G)$
 & Lov\'asz theta number; $\overline\vartheta(G)=\vartheta(\overline G)$.\\
$\F_q$, $\rank_{\mathbb K}M$
 & Field with $q$ elements ($q$ a prime power); rank of $M$ over a field $\mathbb K$.\\
$x\cdot y$, $\langle u,v\rangle$
 & Dot product over the stated finite field; real Euclidean inner product.\\
$\C(G)$, $\Theta(G)$
 & Shannon OR-capacity and ordinary Shannon capacity, with $\C(G)=\Theta(\overline G)$.\\
$f_C(t)$ ($t\ge1$)
 & Infimum of $\C(G\times H)$ with $\C(G),\C(H)\ge t$; see \eqref{eq:extremal}.\\
$\spec$, $\varphi$
 & The OR asymptotic spectrum and a spectral point in it; see \cref{sec:spectral-argument}.\\
$\phif{\mathbb K}(G)$
 & Complementary fractional Haemers bound $\mathcal H_f^{\mathbb K}(\overline G)$ over the field $\mathbb K$.\\
$D(W)$, $R_0(W)$
 & Distinguishability graph; optimal zero-error rate without feedback.\\
$M_0(W,n)$
 & Largest number of messages transmissible with zero error in $n$ channel uses without feedback.\\
$W_{\mathrm{sw}}$
 & Randomly switched channel whose output includes the selected branch label.\\
$C_{\mathrm{Sh}}(W)$
 & Ordinary vanishing-error capacity of a channel.\\
\bottomrule
\end{tabularx}
\par
\endgroup
\addvspace{.8\baselineskip}

\clearpage
\section{Graph products}\label{sec:products}

This appendix compares the categorical, OR and strong graph products
and explains their relation under complementation. These distinctions
underlie the capacity arguments in \cref{sec:simple} and the channel
interpretation in \cref{sec:operational}.

All three graph products have vertex set $V(G)\times V(H)$.
The strong product permits adjacent pairs to share one coordinate,
whereas the categorical product requires adjacency in both.

Complementation relates the OR and strong products through
\[
 \overline{G\OR H}=\overline G\boxtimes\overline H,
 \qquad
 \omega(G^{\OR r})=\alpha\bigl((\overline G)^{\boxtimes r}\bigr).
\]
Indeed, two distinct pairs are nonadjacent in the OR-product precisely
when, in each coordinate, their entries are equal or adjacent in the
complementary graph. Cliques become independent sets under
complementation. Taking $r$th roots and limits in the capacity
definitions gives $\C(G)=\Theta(\overline G)$, as used in
\cref{sec:spectral-argument}; see Simonyi~\cite[Sec.~2]{Simonyi}.

In the distinguishability convention, adjacent inputs can be
distinguished with zero error. Two words are distinguishable if this
holds in at least one coordinate. Thus a clique in $G^{\OR r}$ is a
zero-error code for $r$ channel uses, and $\log_2\C(G)$ is the
asymptotic communication rate. Equivalently, the same code is an
independent set in the strong power of the confusability graph
$\overline G$. The categorical product is the
operation tested by Hedetniemi's identity, as in the framework of
Zhu~\cite[Secs.~2 and~2.2]{ZhuSurvey2025}.
\Cref{fig:products} compares the three products and illustrates the
complement relation.

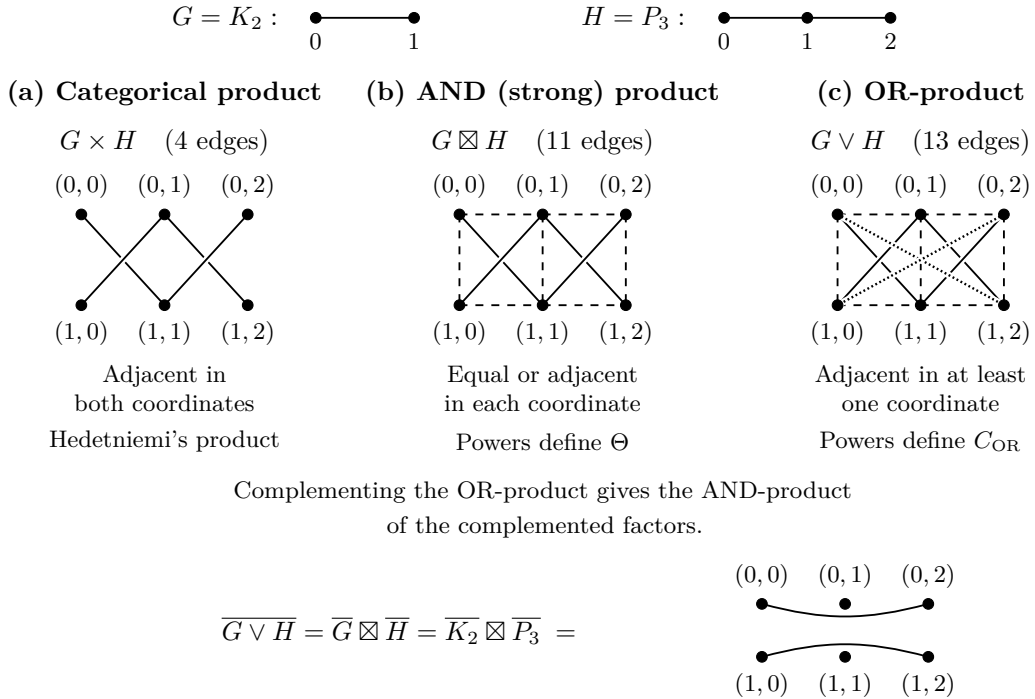
\begin{figure}[!ht]
\centering
\begin{tikzpicture}[
 x=1cm,y=1cm,
 every node/.style={font=\small,align=center,text=black},
 vertex/.style={circle,fill=black,draw=black,inner sep=0pt,
                minimum size=4pt},
 edge/.style={draw=black,line width=.7pt},
 strongextra/.style={edge,dashed},
 orextra/.style={edge,densely dotted,line width=.9pt},
 explanation/.style={font=\footnotesize},
 coordlabel/.style={font=\footnotesize}
]
\node at (.8,3.2) {$G=K_2:$};
\draw[edge] (2,3.2) -- (3.3,3.2);
\node[vertex,label={[coordlabel]below:$0$}] at (2,3.2) {};
\node[vertex,label={[coordlabel]below:$1$}] at (3.3,3.2) {};
\node at (6.2,3.2) {$H=P_3:$};
\draw[edge] (7.4,3.2) -- (8.5,3.2) -- (9.6,3.2);
\node[vertex,label={[coordlabel]below:$0$}] at (7.4,3.2) {};
\node[vertex,label={[coordlabel]below:$1$}] at (8.5,3.2) {};
\node[vertex,label={[coordlabel]below:$2$}] at (9.6,3.2) {};
\begin{scope}[xshift=0cm]
\node[font=\small\bfseries] at (0,2.2) {(a) Categorical product};
\node at (0,1.55) {$G\times H$\quad(4 edges)};
\coordinate (a00) at (-1.1,.6);
\coordinate (a01) at (0,.6);
\coordinate (a02) at (1.1,.6);
\coordinate (a10) at (-1.1,-.6);
\coordinate (a11) at (0,-.6);
\coordinate (a12) at (1.1,-.6);
\draw[edge] (a00) -- (a11);
\draw[edge,preaction={draw=white,line width=2.5pt}] (a01) -- (a10);
\draw[edge] (a01) -- (a12);
\draw[edge,preaction={draw=white,line width=2.5pt}] (a02) -- (a11);
\foreach \v/\pair/\where in {a00/{(0,0)}/above,a01/{(0,1)}/above,
 a02/{(0,2)}/above,a10/{(1,0)}/below,a11/{(1,1)}/below,a12/{(1,2)}/below}{
 \node[vertex,label={[coordlabel]\where:$\pair$}] at (\v) {};
}
\node[explanation] at (0,-1.7) {Adjacent in\\both coordinates};
\node[explanation] at (0,-2.4) {Hedetniemi's product};
\end{scope}
\begin{scope}[xshift=5cm]
\node[font=\small\bfseries] at (0,2.2) {(b) AND (strong) product};
\node at (0,1.55) {$G\boxtimes H$\quad(11 edges)};
\coordinate (b00) at (-1.1,.6);
\coordinate (b01) at (0,.6);
\coordinate (b02) at (1.1,.6);
\coordinate (b10) at (-1.1,-.6);
\coordinate (b11) at (0,-.6);
\coordinate (b12) at (1.1,-.6);
\draw[edge] (b00) -- (b11);
\draw[edge,preaction={draw=white,line width=2.5pt}] (b01) -- (b10);
\draw[edge] (b01) -- (b12);
\draw[edge,preaction={draw=white,line width=2.5pt}] (b02) -- (b11);
\draw[strongextra] (b00) -- (b10);
\draw[strongextra] (b01) -- (b11);
\draw[strongextra] (b02) -- (b12);
\draw[strongextra] (b00) -- (b01);
\draw[strongextra] (b01) -- (b02);
\draw[strongextra] (b10) -- (b11);
\draw[strongextra] (b11) -- (b12);
\foreach \v/\pair/\where in {b00/{(0,0)}/above,b01/{(0,1)}/above,
 b02/{(0,2)}/above,b10/{(1,0)}/below,b11/{(1,1)}/below,b12/{(1,2)}/below}{
 \node[vertex,label={[coordlabel]\where:$\pair$}] at (\v) {};
}
\node[explanation] at (0,-1.7) {Equal or adjacent\\in each coordinate};
\node[explanation] at (0,-2.4) {Powers define $\Theta$};
\end{scope}
\begin{scope}[xshift=10cm]
\node[font=\small\bfseries] at (0,2.2) {(c) OR-product};
\node at (0,1.55) {$G\OR H$\quad(13 edges)};
\coordinate (c00) at (-1.1,.6);
\coordinate (c01) at (0,.6);
\coordinate (c02) at (1.1,.6);
\coordinate (c10) at (-1.1,-.6);
\coordinate (c11) at (0,-.6);
\coordinate (c12) at (1.1,-.6);
\draw[edge] (c00) -- (c11);
\draw[edge,preaction={draw=white,line width=2.5pt}] (c01) -- (c10);
\draw[edge] (c01) -- (c12);
\draw[edge,preaction={draw=white,line width=2.5pt}] (c02) -- (c11);
\draw[strongextra] (c00) -- (c10);
\draw[strongextra] (c01) -- (c11);
\draw[strongextra] (c02) -- (c12);
\draw[strongextra] (c00) -- (c01);
\draw[strongextra] (c01) -- (c02);
\draw[strongextra] (c10) -- (c11);
\draw[strongextra] (c11) -- (c12);
\draw[orextra,preaction={draw=white,line width=2.5pt}] (c00) -- (c12);
\draw[orextra,preaction={draw=white,line width=2.5pt}] (c02) -- (c10);
\foreach \v/\pair/\where in {c00/{(0,0)}/above,c01/{(0,1)}/above,
 c02/{(0,2)}/above,c10/{(1,0)}/below,c11/{(1,1)}/below,c12/{(1,2)}/below}{
 \node[vertex,label={[coordlabel]\where:$\pair$}] at (\v) {};
}
\node[explanation] at (0,-1.7) {Adjacent in at least\\one coordinate};
\node[explanation] at (0,-2.4) {Powers define $\C$};
\end{scope}
\node[explanation,text width=12cm] at (5,-3.3)
 {Complementing the OR-product gives the AND-product\\[2pt]
  of the complemented factors.};
\node at (3.1,-4.9)
 {$\overline{G\OR H}=\overline G\boxtimes\overline H
   =\overline{K_2}\boxtimes\overline{P_3}\;=$};
\begin{scope}[xshift=9cm,yshift=-4.9cm]
\coordinate (d00) at (-1.1,.35);
\coordinate (d01) at (0,.35);
\coordinate (d02) at (1.1,.35);
\coordinate (d10) at (-1.1,-.35);
\coordinate (d11) at (0,-.35);
\coordinate (d12) at (1.1,-.35);
\draw[edge] (d00) to[bend right=15] (d02);
\draw[edge] (d10) to[bend left=15] (d12);
\foreach \v/\pair/\where in {d00/{(0,0)}/above,d01/{(0,1)}/above,
 d02/{(0,2)}/above,d10/{(1,0)}/below,d11/{(1,1)}/below,d12/{(1,2)}/below}{
 \node[vertex,label={[coordlabel]\where:$\pair$}] at (\v) {};
}
\end{scope}
\end{tikzpicture}
\caption{Three products of $G=K_2$ and $H=P_3$ on the same six vertices.
Solid edges form the categorical product; dashed edges are added in the
AND-product and dotted edges in the OR-product. Below, complementing
panel~(c) gives $\overline G\boxtimes\overline H$, with two edges and
two isolated vertices. Its factors are complemented, so it differs from
panel~(b). These graphs illustrate the definitions and are not counterexamples.
OpenAI Codex assisted with preparation of the TikZ source.}
\label{fig:products}
\end{figure}
\FloatBarrier

\section{Capacity and the asymptotic spectrum}\label{sec:spectrum}

This appendix establishes the basic capacity properties and spectral
tools used in the main text. It supplies the proof of
\cref{lem:spectral}, together with the facts about joins and thresholds
needed for the finite and asymptotic results.

For $a_r=\omega(X^{\OR r})$, concatenating every word in a largest
clique of $X^{\OR r}$ with every word in a largest clique of
$X^{\OR s}$ gives $a_ra_s$ pairwise distinguishable words of length
$r+s$. Thus $a_{r+s}\ge a_ra_s$, and $1\le a_r\le |V(X)|^r$.
Fekete's lemma applied to $\log a_r$ therefore gives
\begin{equation}\label{eq:capacity-supremum}
 \C(X)=\lim_{r\to\infty}a_r^{1/r}
       =\sup_{r\ge1}a_r^{1/r}.
\end{equation}
In particular, taking the supremum of finite-block rates agrees with
the limiting rate used in \suppref{sec:channelproof}.

Applying a graph homomorphism coordinatewise gives a homomorphism
between the corresponding OR-powers. A homomorphism into a simple
graph is injective on every clique, so
\[
 X\longrightarrow Y\quad\Longrightarrow\quad\C(X)\le\C(Y).
\]
In particular, $\C(K_m)=m$, and the product projections imply
$\C(X\times Y)\le\min\{\C(X),\C(Y)\}$.
For completeness, this is equality when the smaller factor capacity
is $1$ or $2$~\cite[Rem.~5]{Simonyi}. A graph has capacity $1$
exactly when it is edgeless. If the smaller capacity is $2$, both
factors contain an edge, so their product contains an edge and has
capacity at least $2$. The projection bound completes both cases.

The clique and coloring homomorphisms also give
$\omega(X)\le\C(X)\le\chi(X)$. If $f:X\to Y$ is a homomorphism,
then $x\mapsto(x,f(x))$ is a homomorphism $X\to X\times Y$;
together with projection, this proves the equality recalled in the
introduction. If both factors have capacity equal to their clique
number, pair vertices of their maximum cliques to obtain a clique
of size $\min\{\omega(X),\omega(Y)\}$ in $X\times Y$, proving that
positive case as well.

The asymptotic spectrum and the spectral characterization used below
are introduced in \cref{sec:spectral-argument}. We give the proof of
the spectral lemma here, including the interpolation step based on
Vrana's probabilistic refinement.

\begin{proof}[Proof of \cref{lem:spectral}]
Eq.~\eqref{eq:duality}, including attainment, is Zuiddam's spectrum
theorem~\cite[Thm.~1.1(ii)]{Zuiddam}, translated by complementation to
the OR convention. Since
$\overline{X\sqcup Y}=\overline X\join\overline Y$,
Eq.~\eqref{eq:max} is the same translation of Vrana's maximum
rule for joins~\cite[Prop.~4.3]{Vrana}.

We spell out how the spectral interpolation inequality~\eqref{eq:interpolate}
follows from Vrana's probabilistic refinement. Set $f_j(Y)=\varphi_j(\overline Y)$, so that
$f_j$ is a spectral point in the usual confusability convention.
Let $F_j(Y,P)$ be its logarithmic probabilistic refinement, with
logarithms to base two and $P$ a probability distribution on $V(Y)$.
Let $T_{n,\varepsilon}(P)$ be the set of words in $V(Y)^n$ for which
the frequency of each vertex $v$ differs from $P(v)$ by at most
$\varepsilon$. Writing $Y^{\boxtimes n}[T]$ for the subgraph induced
by a set of words $T$, the refinement is
\[
 F_j(Y,P)=\lim_{\varepsilon\downarrow0}\limsup_{n\to\infty}
 \frac1n\log_2 f_j\bigl(Y^{\boxtimes n}[T_{n,\varepsilon}(P)]\bigr).
\]
It thus measures the asymptotic logarithmic value per coordinate
when the empirical vertex distribution approaches $P$.
Vrana's Theorems~1.1--1.2~\cite{Vrana} identify spectral points with
their logarithmic probabilistic refinements and give the inverse formula
\[
 f_j(Y)=\max_P2^{F_j(Y,P)}.
\]
The same characterization shows that these refinements form a convex
set: its defining inequalities are affine in the refinement.
Consequently $F=\sum_jw_jF_j$ is the refinement of a single spectral
point $f$. This assertion concerns functions of all pairs $(Y,P)$,
so the same $f$ is used for every graph.
Returning to the OR convention, put $\psi(X)=f(\overline X)$. Then
\[
 \psi(X)=\max_P2^{\sum_jw_jF_j(\overline X,P)}
 \le\prod_j\left(\max_P2^{F_j(\overline X,P)}\right)^{w_j}
 =\prod_j\varphi_j(X)^{w_j}.
\]
The same $\psi$ works for all $X$. This does not require the pointwise
geometric mean of the original spectral points to be spectral.
The result numbers in this proof refer to the arXiv versions listed
in the bibliography.
\end{proof}

\Needspace{18\baselineskip}
Complementary fractional Haemers bounds turn matrix-rank estimates
into spectral upper bounds. A $d$-dimensional fitting matrix for a graph $Y$ over a field $\mathbb K$
is a block matrix $M$, indexed by $V(Y)\times V(Y)$, with $d\times d$
identity matrices as diagonal blocks and zero blocks on distinct
nonadjacent pairs. The fractional
Haemers bound, introduced by Blasiak~\cite[Sec.~2.2]{Blasiak},
is the infimum of $\rank_{\mathbb K}(M)/d$ over all such
matrices and all positive integers $d$.

\begin{lemma}[Fractional Haemers spectral bounds]\label{lem:haemers}
For every field $\mathbb K$, the parameter
\[
 \phif{\mathbb K}(X)=\mathcal H_f^{\mathbb K}(\overline X)
\]
belongs to $\spec$. In particular, if a scalar matrix $M$ indexed by
$V(X)$ has diagonal entries $1$ and vanishes on edges of $X$, then
\[
 \phif{\mathbb K}(X)\le\rank_{\mathbb K}(M).
\]
\end{lemma}
\begin{proof}
Spectral membership is the complementary form of the fractional
Haemers result recorded in the arXiv version of~\cite[Sec.~5.2.2]{Zuiddam}; its
multiplicativity is proved in~\cite[Thm.~3]{BukhCox}. The matrix in the statement
is an admissible fitting matrix for $\overline X$ with block size $1$.
\end{proof}

\paragraph{Fractional colourings.}\label{def:fractional-colouring}
Let $\mathcal I(X)$ be the family of independent sets of $X$, that is,
vertex sets containing no edge. A \emph{fractional colouring} assigns
weights $w_I\ge0$ to these sets such that
$\sum_{I\ni v}w_I\ge1$ for every vertex $v$.
The fractional chromatic number $\chi_f(X)$ is the minimum total
weight $\sum_{I\in\mathcal I(X)}w_I$.
A \emph{$d$-fold colouring} with $k$ colours assigns a $d$-element
subset of $\{1,\ldots,k\}$ to
each vertex, with disjoint subsets at adjacent vertices.
Equivalently, $\chi_f(X)=\inf k/d$ over positive integers $k,d$
for which such a colouring exists.
Indeed, such a colouring gives each colour class weight $1/d$.
Conversely, the finite covering linear program has a rational optimal
solution. Clearing its denominators produces colour classes covering
each vertex at least $d$ times; discarding excess colour memberships
at each vertex gives a $d$-fold colouring with the same ratio $k/d$.

For every field $\mathbb K$, these colourings give
\begin{equation}\label{eq:haemers-colouring}
 \phif{\mathbb K}(X)\le\chi_f(X).
\end{equation}
For a $d$-fold colouring with $k$ colours, let $P_v$ have as columns
the $d$ standard coordinate vectors of $\mathbb K^k$ assigned to $v$.
The blocks $P_u^{\mathsf T}P_v$ form a fitting matrix for
$\overline X$ of rank at most $k$, with identity diagonal blocks and
zero blocks on edges of $X$. Taking the infimum of $k/d$ proves
\eqref{eq:haemers-colouring}.

An \emph{automorphism} is a vertex permutation preserving adjacency.
A graph is \emph{vertex-transitive} if an automorphism carries any
specified vertex to any other.
One useful construction averages a nonempty independent $a$-set $I$ in a
vertex-transitive graph $X$ on $n$ vertices over its automorphism
group $\Gamma$. Each vertex occurs in exactly $|\Gamma|a/n$ of the
sets $g(I)$, counted with multiplicity. Giving each occurrence
weight $n/(|\Gamma|a)$ therefore covers every vertex with total
weight one and yields $\chi_f(X)\le n/a$.

For the edgeless graph $I_m$ on $m\ge1$ vertices, every spectral point
satisfies $\varphi(I_m)=1$, since $I_m$ and $K_1$ are homomorphically
equivalent. Two useful consequences of~\eqref{eq:duality} are
\begin{equation}\label{eq:joincopies}
 \C(\underbrace{X\join\cdots\join X}_{k\text{ copies}}\join K_s)
 =k\C(X)+s
 \qquad(k\ge1,\ s\ge0),
\end{equation}
where $K_0$ is omitted, and
\begin{equation}\label{eq:joinindependent}
 \C(X\join I_m)=\C(X)+1\qquad(m\ge1).
\end{equation}
These follow by taking the minimum of $k\varphi(X)+s$ and
$\varphi(X)+1$, respectively. They do not assert additivity of capacity
under arbitrary joins.

The next lemma relates $f_C(m)=m$ to the product identity for pairs
whose smaller capacity is exactly $m$.

\begin{lemma}[Integer thresholds and exact minima]\label{lem:integer-minimum}
For every positive integer $m$, $f_C(m)=m$ if and only if
\[
 \C(G\times H)=m
 \quad\text{whenever}\quad
 \min\{\C(G),\C(H)\}=m.
\]
\end{lemma}
\begin{proof}
If $f_C(m)=m$, its definition gives the required lower bound, and
the coordinate projections give the upper bound.
Conversely, suppose the stated identity holds, and let
$\C(G),\C(H)\ge m$. Applying it to $(G,K_m)$ shows that
$A=G\times K_m$ has capacity $m$. Applying it again to $(A,H)$
gives $\C(A\times H)=m$. The homomorphism
$((g,i),h)\mapsto(g,h)$ from $A\times H$ to $G\times H$ then gives
$\C(G\times H)\ge m$. Taking the infimum yields $f_C(m)\ge m$,
while the pair $(K_m,K_m)$ gives the reverse inequality.
\end{proof}

We also record the threshold facts used in the main text.
The admissible pairs in~\eqref{eq:extremal} form a decreasing family
as $t$ increases, so $f_C$ is nondecreasing; moreover
$1\le f_C(t)\le\lceil t\rceil$, using the pair
$(K_{\lceil t\rceil},K_{\lceil t\rceil})$ for the upper bound.
Thus $f_C(t)\to\infty$ is equivalent to the uniform lower-bound
formulation~\eqref{eq:weak-uniform}. For the converse implication, apply
that formulation with $B+1$ before taking the infimum, to obtain
$f_C(T)\ge B+1>B$.
If the limit is not infinite, monotonicity gives a finite bound $B$
on $f_C$. For each integer $n$, the definition of the infimum supplies
graphs of capacities at least $n$ whose product has capacity less
than $f_C(n)+1\le B+1$. Conversely, any sequence with diverging
factor capacities and uniformly bounded product capacities bounds
$f_C$ by the same constant. If the product capacities fail to diverge,
we may first pass to a subsequence on which they are uniformly bounded.
This gives the equivalence with the sequence formulation in
\cref{conj:weak}. No attainment of the infimum is required.

\Needspace{14\baselineskip}
\section{Product bounds from interpolation}\label{sec:interpolation}

This appendix turns spectral interpolation into upper bounds for
categorical products. These bounds support the
finite and asymptotic constructions in \cref{sec:simple,sec:asymptotic};
we also show how to make the first factor connected.

The categorical product distributes over disjoint union. Fix two
spectral points $h,q\in\spec$ and suppose
$h(A_i\times H)\le h_i$ and $q(A_i\times H)\le q_i$ for $i=0,1$.
For any $\lambda\in[0,1]$, the spectral interpolation
inequality~\eqref{eq:interpolate} supplies a single
$\psi\in\spec$ with $\psi(X)\le h(X)^\lambda q(X)^{1-\lambda}$
for every graph $X$. Combining this with the maximum
rule~\eqref{eq:max} gives the weighted product bound
\begin{equation}\label{eq:weighted-main}
 \C\bigl((A_0\sqcup A_1)\times H\bigr)
 \le\max_{i=0,1}h_i^\lambda q_i^{1-\lambda}
 \le\max_{i=0,1}\{\lambda h_i+(1-\lambda)q_i\}.
\end{equation}
Each $h_i$ and $q_i$ can be bounded by projecting to either factor.
This lets different bounds control different components.

For a complete second factor $K_m$, choose spectral points attaining
the two component capacities.
Projection to $K_m$ supplies a second bound for each component product.
Interpolating with freely chosen weights combines these bounds into
one spectral point.

\begin{proposition}[Weighted two-component product bound]\label{prop:two}
For graphs $A,B$, every integer $m\ge1$ and every $\lambda\in[0,1]$,
\[
 \C\bigl((A\sqcup B)\times K_m\bigr)
 \le\max\bigl\{\C(A)^\lambda m^{1-\lambda},
                 m^\lambda\C(B)^{1-\lambda}\bigr\}.
\]
In particular, taking $\lambda=1/2$ gives
\[
 \C\bigl((A\sqcup B)\times K_m\bigr)
 \le\sqrt{m\max\{\C(A),\C(B)\}}.
\]
\end{proposition}
\begin{proof}
Choose $\varphi_1,\varphi_2\in\spec$ with
$\varphi_1(A)=\C(A)$ and $\varphi_2(B)=\C(B)$.
Apply the spectral interpolation inequality~\eqref{eq:interpolate}
with weights $\lambda,1-\lambda$ to obtain $\psi$.
The two projections imply
\begin{align*}
 \psi(A\times K_m)&\le
 \varphi_1(A\times K_m)^\lambda\varphi_2(A\times K_m)^{1-\lambda}
 \le\C(A)^\lambda m^{1-\lambda},\\
 \psi(B\times K_m)&\le
 \varphi_1(B\times K_m)^\lambda\varphi_2(B\times K_m)^{1-\lambda}
 \le m^\lambda\C(B)^{1-\lambda}.
\end{align*}
Since the categorical product distributes over disjoint union,
\[
 \C((A\sqcup B)\times K_m)
 \le\psi((A\sqcup B)\times K_m)
 =\max\{\psi(A\times K_m),\psi(B\times K_m)\}.
\]
The preceding bounds prove the weighted inequality, and $\lambda=1/2$
gives the square-root bound. The same spectral point $\psi$ must
control both components for this step to apply.
\end{proof}

For the asymptotic construction, each component may have its own
exceptional spectral points, where the common bound is not assumed.
Projection to $K_m$ controls these exceptions.

\begin{lemma}[Several spectral bounds]\label{lem:many}
Let $A_1,\ldots,A_N$ be graphs and
$\varphi_1,\ldots,\varphi_k\in\spec$. Suppose $1\le h<k$, $D\ge1$,
and for each $i$ there is a set $S_i\subseteq\{1,\ldots,k\}$ of
cardinality exactly $h$ such that
\[
 \varphi_j(A_i)\le D\qquad(j\notin S_i).
\]
Then, for every integer $m\ge1$,
\[
 \C\left(\left(\bigsqcup_{i=1}^N A_i\right)\times K_m\right)
 \le D^{1-h/k}m^{h/k}.
\]
\end{lemma}
\begin{proof}
Apply the spectral interpolation inequality~\eqref{eq:interpolate}
to the $k$ spectral points with equal weights to obtain $\psi$.
On $A_i\times K_m$, projection to $A_i$ gives a bound $D$ for each of the
$k-h$ indices outside $S_i$. Projection to $K_m$ gives a bound $m$
for each of the remaining $h$ indices. Hence, for every $i$,
\[
 \psi(A_i\times K_m)
 \le\prod_{j=1}^k\varphi_j(A_i\times K_m)^{1/k}
 \le(D^{k-h}m^h)^{1/k}.
\]
The maximum rule~\eqref{eq:max} gives the same bound on the disjoint
union of these products, and $\C\le\psi$ proves the claim.
\end{proof}

Our constructions initially give disjoint unions. Adding one universal
vertex makes them connected, at an additive cost of at most one in
the product bound.

\begin{lemma}[Adding a universal vertex]\label{lem:connected}
For every graph $X$ and every integer $m\ge1$, the connected graph
$X^+=X\join K_1$ satisfies
\[
 \C(X^+)=\C(X)+1,
 \qquad
 \C(X^+\times K_m)\le\C(X\times K_m)+1.
\]
\end{lemma}
\begin{proof}
The equality is~\eqref{eq:joincopies}. The graph $X^+\times K_m$ is a
spanning subgraph of $(X\times K_m)\join I_m$. The vertices above the
new universal vertex form an independent set of size $m$. Monotonicity
and~\eqref{eq:joinindependent} prove the upper bound.
\end{proof}

\Needspace{15\baselineskip}
\section{Proof of Theorem~\ref*{thm:finite}: counterexamples with a complete factor}\label{sec:finiteproof}

This appendix supplies the proofs for the finite counterexamples with
a complete factor in \cref{sec:simple}. We first establish the
binary graph estimates shared by the constructions, then prove the
simple example, \cref{thm:finite} and \cref{cor:spectral}.

\subsection{Binary orthogonality graphs}\label[subsection]{sec:binaryproof}

For even $r\ge2$, let $U_r$ be the graph on the odd-weight vectors of
$\F_2^r$, with distinct $x,y$ adjacent when $x\cdot y=0$ in $\F_2$.
Here the \emph{Hamming weight} of a vector is its number of nonzero
coordinates.
The introductory example uses $r=8$, whereas the construction for
\cref{thm:finite} and the real-threshold refinements use $r=6$.
The next lemma determines the
capacities of $U_r$ and its complement by matching explicit cliques
with upper bounds valid for every OR-power.
We also record the theta bounds needed for spectral interpolation.
In our convention, Lov\'asz's orthonormal-representation formula
bounds $\overline\vartheta(X)=\vartheta(\overline X)$ by
$\max_x|\langle u,v_x\rangle|^{-2}$ whenever $u,v_x$ are unit vectors
and $v_x\perp v_y$ for every edge $xy$ of $X$~\cite{Lovasz}.
The fixed unit reference vector $u$ is called the \emph{handle}.

\begin{lemma}[Binary orthogonality graphs]\label{lem:binary}
For every even integer $r\ge2$,
\[
 \C(U_r)=r,
 \qquad \C(\overline{U_r})=\vartheta(U_r)=2^{r/2-1}.
\]
\end{lemma}
\begin{proof}
For $U_r$, we apply Haemers' matrix-rank bound~\cite{Haemers}.
The binary matrix $M_{xy}=x\cdot y$ has diagonal $1$, vanishes on
edges of $U_r$, and has rank at most $r$. A clique in $U_r^{\OR t}$
indexes an identity principal submatrix of $M^{\otimes t}$, so its
size is at most $r^t$. The $r$ coordinate vectors form a clique in
$U_r$, proving the first equality.

For the complementary graph, we use Lov\'asz's orthonormal representation
method~\cite{Lovasz}. We construct unit vectors orthogonal on edges,
all with the same overlap with a fixed unit vector.
Tensoring these vectors will bound cliques in every OR-power.
Write $V=V(U_r)$, $n=2^{r-1}$ and $q=2^{r/2-1}$, so $n=2q^2$.
Since $r$ is even, addition of the all-ones vector $\mathbf1$ pairs
the elements of $V$. In $\mathbb R^V$, let $b_x$ be the coordinate
vector indexed by $x$, and define $h_x(z)=(-1)^{x\cdot z}$ for $z\in V$.
With $\delta$ denoting the Kronecker delta, we have
\[
 \langle h_x,h_y\rangle
 =\sum_{z\in V}(-1)^{(x+y)\cdot z}
 =n\bigl(\delta_{xy}-\delta_{y,x+\mathbf1}\bigr).
\]
The sum is $n$ when $x=y$ and $-n$ when $x+y=\mathbf1$.
Otherwise, choose $i,j$ with $(x+y)_i\ne(x+y)_j$ and pair the
summands at $z$ and $z+e_i+e_j$; these have opposite signs and cancel.

Define
\[
 w_x=\left(q^{-1/2},
       \frac{b_x-b_{x+\mathbf1}}2+\frac{h_x}{2q}\right),
 \qquad u=(1,0,\ldots,0).
\]
Since $h_y(x+\mathbf1)=-h_y(x)$, expansion of the inner product gives
\begin{align*}
 \langle w_x,w_y\rangle
 &=\frac1q+\frac12(\delta_{xy}-\delta_{y,x+\mathbf1})
   +\frac{(-1)^{x\cdot y}}q
   +\frac{n}{4q^2}(\delta_{xy}-\delta_{y,x+\mathbf1})\\
 &=\delta_{xy}-\delta_{y,x+\mathbf1}
   +\frac{1+(-1)^{x\cdot y}}q.
\end{align*}
Thus every $w_x$ is a unit vector and
$|\langle u,w_x\rangle|^2=1/q$. If $x,y$ are adjacent in
$\overline{U_r}$, then $x\cdot y=1$ and $y\notin\{x,x+\mathbf1\}$,
since $x\cdot(x+\mathbf1)=0$. The displayed formula gives
$w_x\perp w_y$.
The orthonormal-representation formula then gives
$\vartheta(U_r)\le q$.

A clique of size $a$ in $\overline{U_r}^{\OR t}$ gives
$a$ orthonormal tensor vectors, each with squared overlap $q^{-t}$
with $u^{\otimes t}$. Bessel's inequality gives
$a q^{-t}\le\|u^{\otimes t}\|^2=1$, so $\C(\overline{U_r})\le q$.

\Needspace{6\baselineskip}
For the matching lower bound, the $q$ vectors
\[
 e_r+\operatorname{span}_{\F_2}
 \{e_1+e_2,e_3+e_4,\ldots,e_{r-3}+e_{r-2}\}
\]
have pairwise dot product $1$ and form a clique in $\overline{U_r}$.
For $r=2$, the span is $\{0\}$.
\end{proof}

\subsection{The simple counterexample}\label[subsection]{sec:nineproof}

\begin{proposition}[The simple counterexample]\label{prop:nine}
Let $F_9=U_8\sqcup\overline{U_8}$. Then
\[
 \C(F_9)=\sqrt{128}>9=\C(K_9),\qquad
 \C(F_9\times K_9)\le\sqrt{72}<9.
\]
\end{proposition}

\begin{proof}[Proof of \cref{prop:nine} and \eqref{eq:nine-base}]
By \cref{lem:binary},
$\C(U_8)=\C(\overline{U_8})=8$. For each vertex $x$ of $U_8$, take the
word $(x_{U_8},x_{\overline{U_8}})$ in $(U_8\sqcup\overline{U_8})^{\OR2}$.
For distinct $x,y$, either $xy\in E(U_8)$ or $xy\in E(\overline{U_8})$,
so the 128 words form a clique. Hence
$\C(F_9)\ge\sqrt{128}>9$.

For the matching upper bound, first note that $U_8$ is vertex-transitive.
Given an odd-weight vector $x$, choose $i$ with $x_i=0$ and put
$v=x+e_i$. Since $v\cdot v=0$, the map
$z\mapsto z+(z\cdot v)v$ preserves dot products and sends $e_i$ to $x$.
Together with coordinate permutations, these maps prove transitivity.
Put $h=\phif{\F_2}$ and $q(X)=\vartheta(\overline X)$.
The rank bound and the eight-clique in \cref{lem:binary} give $h(U_8)=8$,
and the same lemma gives $q(\overline{U_8})=8$.
Averaging that eight-clique, viewed as an independent set of
$\overline{U_8}$, over automorphisms gives a fractional colouring of
weight $128/8=16$, as explained in \suppref{sec:spectrum}. Thus
\eqref{eq:haemers-colouring} gives
$h(\overline{U_8})\le\chi_f(\overline{U_8})\le16$.
Lov\'asz's vertex-transitive identity
$\vartheta(U_8)\vartheta(\overline{U_8})=128$~\cite[Thm.~8]{Lovasz}
also gives $q(U_8)=16$. Applying the spectral interpolation
inequality~\eqref{eq:interpolate} to $h,q$ with equal weights
yields a single $\psi\in\spec$ satisfying
\[
 \psi(U_8)\le\sqrt{8\cdot16},\qquad
 \psi(\overline{U_8})\le\sqrt{16\cdot8}.
\]
The maximum rule now gives
$\C(F_9)\le\psi(F_9)\le\sqrt{128}$, proving equality.
Finally, the equal-weight bound in \cref{prop:two} gives
$\C(F_9\times K_9)\le\sqrt{72}<9$.
\end{proof}

\begin{remark}[Choice of the simple example]
The graph $F_9$ illustrates the argument with two components of capacity
$8$ and equal interpolation weights. The construction in
\cref{sec:seven} gives a smaller first factor at $m=9$ and extends the
range of complete factors to every integer $m\ge7$.
\end{remark}

\subsection{The complete-factor family and its spectral consequence}\label[subsection]{sec:seven}

We begin with a construction whose smaller factor capacity is seven
and obtain the remaining integer cases by adding cliques. This proves
\cref{thm:finite}, and the same spectral point yields \cref{cor:spectral}.

Put $A=U_6$, $D=\overline A$, and let $Q$ be the 5-regular Clebsch
graph on $\F_2^4$, with adjacency differences
$\{e_1,e_2,e_3,e_4,\mathbf1\}$. Set $B=D\join Q$ and
$X_7=A\sqcup B$. The graph $X_7$ has 80 vertices. We will show that
\begin{equation}\label{eq:rho}
 \C(X_7)\ge\rho:=\frac{s+\sqrt{s^2+128}}2>7,
 \qquad s=16^{1/3},
\end{equation}
while both components have capacity at most $20/3<7$.
Weighted spectral interpolation gives the counterexample with $K_7$;
joining the same complete graph to both components gives the full statement.

\begin{lemma}[Clebsch graph]\label{lem:clebsch}
Let $Q$ be the graph on $\F_2^4$ in which $x$ and $y$ are adjacent
when $x+y\in S:=\{e_1,e_2,e_3,e_4,\mathbf1\}$. Then
\[
 16^{1/3}\le\C(Q)\le\overline\vartheta(Q)\le\frac83.
\]
Moreover, $Q$ admits unit vectors orthogonal on edges and a unit
handle whose squared overlap with every vertex vector is $3/8$.
\end{lemma}
\begin{proof}
For the lower bound~\cite{GreenwoodGleason,Simonyi}, identify $\F_2^4$ with $\F_{16}$ using the
basis $1,\beta,\beta^2,\beta^3$, where $\beta$ has multiplicative
order five. Its minimal polynomial is
$z^4+z^3+z^2+z+1$, so the connection set $S$ becomes the
multiplicative subgroup $H=\{1,\beta,\beta^2,\beta^3,\beta^4\}$.
Choose $a\in\F_{16}^{\times}\setminus H$.
The three cosets $H,aH,a^2H$ partition $\F_{16}^{\times}$.
Consequently, for every nonzero $d$, exactly one of $d,ad,a^2d$
belongs to $H$. The sixteen words
\[
 (x,ax,a^2x),\qquad x\in\F_{16},
\]
therefore form a clique in $Q^{\OR3}$, proving
$\C(Q)\ge16^{1/3}$.

For the upper bound, let $D$ be the adjacency matrix of $Q$.
The sixteen characters $x\mapsto(-1)^{u\cdot x}$, indexed by
$u\in\F_2^4$, form an orthogonal basis of $\mathbb R^{\F_2^4}$.
Each is an eigenvector of $D$
with eigenvalue $4-2|u|+(-1)^{|u|}$, where $|u|$ is the Hamming
weight of $u$. Thus the eigenvalues are $5$, $1$, and $-3$, with
the eigenvalue $5$ occurring only for the constant character.
Writing $J$ for the all-ones matrix, the matrix
\[
 T=I+\frac12(J-I-D)
\]
has eigenvalues $6$, $0$, and $2$, and is therefore positive
semidefinite. It has diagonal entries one, vanishes on edges of $Q$,
and has row sum six. Let $(v_x)_{x\in\F_2^4}$ be unit vectors with
Gram matrix $T$, meaning $T_{xy}=\langle v_x,v_y\rangle$. Such vectors
exist because $T$ is positive semidefinite with diagonal entries one.
The vector
\[
 h=\frac1{\sqrt{96}}\sum_xv_x
\]
is a unit vector satisfying $|\langle h,v_x\rangle|^2=3/8$ for
every $x$. Thus the orthonormal-representation formula gives
$\overline\vartheta(Q)\le8/3$. Tensoring this representation and applying Bessel's
inequality, as in \cref{lem:binary}, gives
$\omega(Q^{\OR t})\le(8/3)^t$.
\end{proof}

The next lemma explains how the complementary diagonal code and the
additional joined graph increase the capacity of the disjoint union.
It uses additivity of spectral points under join, without assuming
that capacity itself is additive.

\begin{lemma}[A joined complementary component]\label{lem:joined-complement}
Let $A$ be a graph on $n$ vertices, and let $Q$ be a graph with
$\C(Q)\ge s>0$. Then
\[
 \C\bigl(A\sqcup(\overline A\join Q)\bigr)
 \ge\frac{s+\sqrt{s^2+4n}}2.
\]
\end{lemma}
\begin{proof}
For any $\varphi\in\spec$, put $x=\varphi(A)$ and
$y=\varphi(\overline A)$. The $n$ diagonal pairs $(v,v)$ form a clique
in $A\OR\overline A$, so multiplicativity and monotonicity give
$xy\ge n$. By additivity under join and~\eqref{eq:max},
\[
 \varphi\bigl(A\sqcup(\overline A\join Q)\bigr)
 =\max\{x,y+\varphi(Q)\}\ge\max\{x,y+s\}.
\]
Writing the last maximum as $r$ gives $x\le r$, $y\le r-s$,
and therefore $n\le xy\le r(r-s)$. Solving this quadratic inequality
and taking the minimum over $\varphi$ proves the claim.
\end{proof}

\Needspace{12\baselineskip}
\begin{proof}[Proof of \cref{thm:finite}]
Set
\[
 A=U_6,\qquad B=\overline A\join Q,\qquad X_7=A\sqcup B.
\]
By \cref{lem:binary}, $\C(A)=6$ and $\C(\overline A)=4$.
Its proof supplies a representation of $\overline A$ with squared
handle overlap $1/4$. Combine this representation with that of $Q$
from \cref{lem:clebsch} in orthogonal summands.
For handles $u$ and $h$ in the respective summands, use
\[
 c=\sqrt{\frac{4}{20/3}}\,u\;\oplus\;
   \sqrt{\frac{8/3}{20/3}}\,h.
\]
This is a unit handle with squared overlap $3/20$ with every vertex
vector. The summands are orthogonal, so the representation also
vanishes on all edges added by the join. Bessel's inequality gives
\begin{equation}\label{eq:seven-components}
 \max\{\C(A),\C(B)\}\le\frac{20}{3}<7.
\end{equation}

The graph $A$ has 32 vertices. Applying \cref{lem:joined-complement}
with $s=16^{1/3}$ gives
\[
 \C(X_7)\ge\frac{s+\sqrt{s^2+128}}2=\rho>7,
\]
where $\rho$ is defined in~\eqref{eq:rho}.
For the strict comparison, $s>17/7$ because
$16\cdot7^3=5488>4913=17^3$; hence
$7^2-7s-32=17-7s<0$.

Following the weighted argument of \cref{prop:two}, choose spectral
points $\varphi_A,\varphi_B$ attaining the capacities of the base
components $A,B$, respectively. Interpolate them with
weights $1/4,3/4$ in~\eqref{eq:interpolate} to obtain a single
$\psi_7\in\spec$, which we fix for the whole family.
For an integer $m\ge7$, let $k=m-7$ and set
\[
 A_m=A\join K_k,\qquad B_m=B\join K_k,
 \qquad X_m=A_m\sqcup B_m,
\]
omitting $K_0$ when $k=0$.
The spectral rules give
$\varphi(X_m)=\varphi(X_7)+k$ for every $\varphi\in\spec$;
therefore $\C(X_m)=\C(X_7)+k>m$.
Eq.~\eqref{eq:joincopies} and
\eqref{eq:seven-components} give
\[
 \C(A_m)=m-1,\qquad \C(B_m)\le m-\frac13.
\]
Join additivity also gives $\varphi_A(A_m)=m-1$ and
$\varphi_B(B_m)\le m-\tfrac13$. The coordinate projections and
the weighted arithmetic--geometric mean inequality therefore give
\[
\begin{aligned}
 \psi_7(A_m\times K_m)
 &\le(m-1)^{1/4}m^{3/4}
 \le\tfrac14(m-1)+\tfrac34m=m-\tfrac14,\\
 \psi_7(B_m\times K_m)
 &\le m^{1/4}(m-\tfrac13)^{3/4}
 \le\tfrac14m+\tfrac34(m-\tfrac13)=m-\tfrac14.
\end{aligned}
\]
Since the categorical product distributes over disjoint union, the
maximum rule~\eqref{eq:max} yields
\[
 \C(X_m\times K_m)\le\psi_7(X_m\times K_m)
 \le m-\tfrac14<m
 =\min\{\C(X_m),\C(K_m)\}.
\]
Finally, $|V(A)|=32$ and $|V(Q)|=16$, so
$|V(X_m)|=32+(32+16)+2k=2m+66$.
Taking $G=X_m$ proves \cref{thm:finite}.
\end{proof}

\begin{proof}[Proof of \cref{cor:spectral}]
The preceding proof constructs a single spectral point $\psi=\psi_7$,
independent of $m$, such that $G_m=X_m$ satisfies
\[
 \psi(G_m\times K_m)\le m-\tfrac14<m.
\]
Every spectral point bounds capacity from above, so
$\psi(G_m)\ge\C(G_m)>m$. Normalization and join additivity give
$\psi(K_m)=m$, proving the claim.
\end{proof}

\input{subsix_supplement.tex}
\input{threshold_supplement.tex}

\section{Proof of Theorem~\ref*{thm:asymptotic}: asymptotic separation}\label{sec:asymproof}

This appendix proves the asymptotic separation in \cref{thm:asymptotic}.
We construct graphs using several finite fields and choose their
parameters so that both factor capacities grow while the product
capacity is subpolynomial in the smaller factor capacity.

The next lemma provides both a large code and spectral bounds over
several fields. Its polynomial-rank argument adapts Alon's finite-field
method~\cite[Secs.~3--4]{Alon}. The edge sets need only cover the
complete graph, so an edge may belong to more than one graph.

\begin{lemma}[A finite-field covering construction]\label{lem:partition}
Let $P\ge2$ be an integer, let $q_1,\ldots,q_k$ be distinct primes
in $(P,2P)$, and let $1\le h<k$. Set
\[
 L=P^{h+1},\qquad d=PL,\qquad n=\binom dL,\qquad
 N=\binom kh,\qquad D=d^{2P}.
\]
There are graphs $A_S$, indexed by the $h$-element subsets
$S\subseteq\{1,\ldots,k\}$, on a common set of $n$ vertices, whose
edges cover $K_n$ and which satisfy
\begin{equation}\label{eq:fields}
 \phif{\F_{q_j}}(A_S)\le D\qquad(j\notin S).
\end{equation}
For their disjoint union $F=\bigsqcup_{|S|=h}A_S$, one has
\begin{equation}\label{eq:partitionbounds}
 \C(F)\ge n^{1/N},
 \qquad \C(F\times K_m)\le D^{1-h/k}m^{h/k}
 \quad(m\ge1).
\end{equation}
\end{lemma}
\begin{proof}
Take the common vertex set to be $\binom{[d]}L$. For distinct
vertices $X,Y$, put
\[
 B(X,Y)=\{j:q_j\mid L-|X\cap Y|\},
 \qquad XY\in E(A_S)\ \Longleftrightarrow\ B(X,Y)\subseteq S.
\]
The positive integer $L-|X\cap Y|\le L$ has at most $h$ of the
selected primes as divisors, since any $h+1$ of them have product
strictly greater than $P^{h+1}=L$. Thus $B(X,Y)$ is contained in
some $h$-element set $S$, proving the covering condition.

Over $\F_{q_j}$ define
\[
 M_j(X,Y)=1-(|X\cap Y|-L)^{q_j-1}.
\]
Its diagonal entries are $1$. If $XY\in E(A_S)$ and $j\notin S$,
then $q_j$ does not divide $L-|X\cap Y|$, so Fermat's theorem gives
$M_j(X,Y)=0$. Each row is a polynomial of degree at most $q_j-1$
in the incidence vector of $Y$, the zero-one vector recording
membership in $Y$. On zero-one vectors, replacing every
positive power of a variable by that variable leaves its value
unchanged. The row space therefore lies in the common span of the
evaluation vectors of squarefree monomials of degree at most $q_j-1$.
Explicitly, these are the vectors with entries
$Y\mapsto\mathbf1_{\{T\subseteq Y\}}$, indexed by subsets
$T\subseteq[d]$ with $|T|\le q_j-1$.
Since $q_j<2P\le d$, this gives
\[
 \rank_{\F_{q_j}}M_j
 \le\sum_{\ell=0}^{q_j-1}\binom d\ell
 \le q_jd^{q_j-1}\le d^{q_j}\le D.
\]
\Cref{lem:haemers} proves~\eqref{eq:fields}.

For each vertex $X$, form the $N$-tuple of its copies in all graphs
$A_S$. The covering condition makes these $n$ tuples a clique in $F^{\OR N}$,
giving the lower bound in~\eqref{eq:partitionbounds}. The upper bound
is \cref{lem:many}, with exceptional set $S$ for graph $A_S$.
\end{proof}

\Needspace{12\baselineskip}
\begin{proof}[Proof of \cref{thm:asymptotic}]
Let $x=\log m$. For every sufficiently large integer $m$, choose
\[
 P=\lceil\sqrt{x}\rceil,\qquad
 k=\left\lfloor\frac{P}{3\log P}\right\rfloor,\qquad
 h=\left\lceil\frac{\log P}{\log\log P}\right\rceil.
\]
Writing $\pi(u)$ for the number of primes at most $u$, the prime
number theorem~\cite{Zagier} gives
$\pi(2P)-\pi(P)\sim P/\log P$.
The number of primes in the open interval $(P,2P)$ differs from this
count by at most one, so it eventually exceeds $k$.
Also $2\le h<k$ eventually. Apply \cref{lem:partition}.

To bound the factor capacity from below, use
$\binom dL\ge(d/L)^L=P^L$ and $N\le k^h$.
Since $k\le P/\log P$ and $(\log P)^h\ge P$, we obtain
\[
 \frac{\log n}{N}
 \ge\frac{P^{h+1}\log P}{k^h}
 \ge P\log P\,(\log P)^h
 \ge P^2\log P\ge x
\]
for sufficiently large $P$. Hence $\C(F)\ge m$.

To bound the product capacity from above, note that $h\ge2$ gives
$\log D=2P(h+2)\log P\le4Ph\log P$.
Also $k\ge P/(4\log P)$ eventually and $x\le P^2$, so
$xh/k\le4Ph\log P$. Therefore
\begin{align}
 \log\C(F\times K_m)
 &\le(1-h/k)\log D+(h/k)x\notag\\
 &\le8Ph\log P
 =O\!\left(\sqrt{x}\,
             \frac{(\log x)^2}{\log\log x}\right).
 \label{eq:optimized}
\end{align}
All constants are absolute. Set $G_m=F\join K_1$.
By \cref{lem:connected}, it is connected, has capacity at least $m+1$,
and its product capacity is at most $\C(F\times K_m)+1$.
Absorbing this additive $1$ into the constant in~\eqref{eq:optimized}
proves~\eqref{eq:asym-factor}--\eqref{eq:asym-product}.

For real $t$, take $m=\lceil t\rceil$. The pair $(G_m,K_m)$ is
admissible in~\eqref{eq:extremal}, and
$\log\lceil t\rceil/\log t\to1$. Thus the same explicit upper bound
holds for $f_C(t)$ after changing the constant. Since
$(\log\log t)^2/\log\log\log t=o((\log t)^\delta)$ for every
$\delta>0$, this proves~\eqref{eq:clean-asym}.
Finally $f_C(t)\ge1$ and $\log f_C(t)=o(\log t)$ give both limits
in~\eqref{eq:limits}.
\end{proof}

\Needspace{14\baselineskip}
\section{Proof of Proposition~\ref*{prop:switch} and its rate consequences}\label{sec:channelproof}

This appendix proves \cref{prop:switch,cor:finite-rates,cor:asym-rates}
from the channel interpretation in \cref{sec:operational}.
We realize graphs as channels, identify the
graph of the switched channel, and derive the stated zero-error rate
bounds and the comparison with vanishing-error capacity.

\paragraph{Zero-error codes and graph realizations.}
We give the standard support argument for~\eqref{eq:channel-rate}
and the realization fact used in \cref{prop:switch}.
For a finite discrete memoryless channel $W$, put
$\mathcal O_W(a)=\{y:W(y\mid a)>0\}$. Distinct inputs $a,a'$ are adjacent
in $D(W)$ if $\mathcal O_W(a)\cap\mathcal O_W(a')=\varnothing$.
By memorylessness, the output support of an input word is the Cartesian
product of its coordinate supports. Two such products are disjoint
if and only if the supports are disjoint in at least one coordinate.
Thus the largest zero-error
code of block length $n$ has size $\omega(D(W)^{\OR n})$, proving
$R_0(W)=\log_2\C(D(W))$ by the supremum characterization of
Shannon OR-capacity. This proves~\eqref{eq:channel-rate}.

Every finite graph $G$ is a distinguishability graph. Assign each vertex
a private output symbol, and each unordered pair of distinct nonadjacent
vertices an additional output symbol shared by that pair. On input $a$,
choose uniformly among its private symbol and the shared symbols
assigned to pairs containing $a$. Two distinct inputs share a possible
output exactly when they are nonadjacent in $G$.

\begin{proof}[Proof of \cref{prop:switch}]
Since $p$ and $1-p$ are positive, the output support of $(a,b)$ is
the disjoint union
\[
 \bigl(\{0\}\times\mathcal O_{W_G}(a)\bigr)
 \cup\bigl(\{1\}\times\mathcal O_{W_H}(b)\bigr).
\]
The labels $0$ and $1$ make outputs from different branches distinct.
The supports for $(a,b)$ and $(a',b')$ are therefore disjoint if and
only if
\[
 \mathcal O_{W_G}(a)\cap\mathcal O_{W_G}(a')=\varnothing
 \quad\text{and}\quad
 \mathcal O_{W_H}(b)\cap\mathcal O_{W_H}(b')=\varnothing.
\]
These conditions are exactly $aa'\in E(G)$ and $bb'\in E(H)$,
which prove~\eqref{eq:switch-graph}. Applying~\eqref{eq:channel-rate}
to the memoryless channel $W_{\mathrm{sw}}$ gives the claimed
block-code size and zero-error rate. Its output supports do not
depend on the value of $p\in(0,1)$, so neither quantity depends on $p$.

If \cref{conj:strong} holds, taking base-two logarithms and applying
\eqref{eq:channel-rate} proves~\eqref{eq:switch-benchmark} for every
channel pair. Conversely, realize any finite graphs $G,H$ by channels
using the construction above. If~\eqref{eq:switch-benchmark} holds
for every pair, the rate identity just proved gives
\[
 \log_2\C(G\times H)
 =\min\{\log_2\C(G),\log_2\C(H)\}.
\]
Exponentiating proves \cref{conj:strong}. All capacities are at least
one, so these logarithms are well-defined.
\end{proof}

\paragraph{Comparison with vanishing-error capacity.}
For completeness, we prove the comparison stated in
\cref{sec:operational} for the same channel $W_{\mathrm{sw}}$.
If $C_{\mathrm{Sh}}$ denotes the ordinary vanishing-error capacity in bits
per use, then
\begin{equation}\label{eq:ordinary-switch}
 C_{\mathrm{Sh}}(W_{\mathrm{sw}})
 =pC_{\mathrm{Sh}}(W_G)+(1-p)C_{\mathrm{Sh}}(W_H).
\end{equation}
\begin{proof}
For a finite discrete memoryless channel, the channel-coding theorem
gives $C_{\mathrm{Sh}}(W)=\max_{P_X}I(X;Y)$~\cite[Sec.~3.1]{ElGamalKim}.
Here $I(X;Y)$ is the mutual information in bits, defined by
\[
 I(X;Y)=\sum_{x,y}P_{XY}(x,y)
 \log_2\frac{P_{XY}(x,y)}{P_X(x)P_Y(y)},
\]
with zero-probability terms omitted; conditional mutual information
is the average of this quantity over the conditioning variable.
Write $A,B$ for the input random variables and
$S$ for the independent switch, with $\Pr(S=0)=p$. The mutual
information between input and tagged output is
\begin{align*}
 I(A,B;S,Y_S)
 &=I(A,B;S)+I(A,B;Y_S\mid S)\\
 &=pI(A;Y_G)+(1-p)I(B;Y_H).
\end{align*}
Here $S$ is independent of the inputs, and each component output
depends only on its own input. This identity holds even when $A$
and $B$ are correlated.
Each term is bounded by the corresponding component capacity. Choosing
independent capacity-achieving input distributions for $A$ and $B$
attains both bounds simultaneously, proving~\eqref{eq:ordinary-switch}.
\end{proof}

\Needspace{20\baselineskip}
The operational consequences now follow by applying \cref{prop:switch}
and taking base-two logarithms of the bounds on graph capacity,
together with the explicit codes for the finite constructions.
For every finite channel $W$ and every $n\ge1$, the definition of
$R_0(W)$ gives
\begin{equation}\label{eq:block-capacity-bound}
 M_0(W,n)\le2^{nR_0(W)}=\C(D(W))^n.
\end{equation}

\begin{proof}[Proof of~\eqref{eq:finite-rates} and
\eqref{eq:finite-two-use}]
Let $W_F$ realize $F_9$, and let $W_9$ be a noiseless nine-symbol
channel. \Cref{prop:nine} and~\eqref{eq:nine-base} give
\[
 \C(F_9)=\sqrt{128}>9,\qquad
 \C(F_9\times K_9)\le\sqrt{72}<9.
\]
Taking base-two logarithms and applying \cref{prop:switch} proves
\eqref{eq:finite-rates}.

The explicit 128-word code in \cref{sec:simple} gives
$M_0(W_F,2)\ge128$, and~\eqref{eq:block-capacity-bound} gives the
matching upper bound. The noiseless channel satisfies
$M_0(W_9,2)=9^2=81$. Finally, \cref{prop:switch}
and~\eqref{eq:block-capacity-bound} give
\[
 M_0(W_{\mathrm{sw}},2)\le\C(F_9\times K_9)^2\le72,
\]
proving~\eqref{eq:finite-two-use}.
\end{proof}

\begin{proof}[Proof of \cref{cor:finite-rates}]
For part~(1), realize the graph $G$ from \cref{thm:finite} by a
channel $W_G$, and take $W_m$ to be noiseless on $m$ symbols.
Taking base-two logarithms of the theorem's bounds and applying
\cref{prop:switch} gives the claimed inequalities.

For part~(2), realize the graphs constructed for \cref{thm:sharper}(1)
by channels $W_G,W_H$. The rate bounds again follow by taking logarithms
and applying \cref{prop:switch}. The explicit codes in
\cref{subsix:sec:graphs} give 27 messages in two uses of each
individual channel. On the other hand,
\eqref{eq:block-capacity-bound} gives
\[
 M_0(W_{\mathrm{sw}},2)\le\C(G\times H)^2<27.
\]
Since the number of messages is an integer, it is at most 26.

For part~(3), put $t=2^r\ge6.0101$ and realize the graphs supplied
by \cref{thm:sharper}(2). Taking logarithms of
$\C(G\times H)<t<\min\{\C(G),\C(H)\}$ gives the claimed rate
separation. In every part, the graph realizations are independent
of $p$, and \cref{prop:switch} applies to the same channels for
every $p\in(0,1)$.
\end{proof}

\begin{proof}[Proof of \cref{cor:asym-rates}]
For each sufficiently large $m$, realize the graph $G_m$ from
\cref{thm:asymptotic} by a channel $W_{G_m}$, and let $W_m$ be a
noiseless $m$-symbol channel. By \cref{prop:switch}, taking base-two
logarithms in~\eqref{eq:asym-product} and writing $R_m=\log_2 m$ gives,
for every $p\in(0,1)$,
\[
 R_0(W_{\mathrm{sw},m})
 =O\!\left(\sqrt{R_m}\,
       \frac{(\log R_m)^2}{\log\log R_m}\right).
\]
The factor multiplying $\sqrt{R_m}$ is $R_m^{o(1)}$; hence
\eqref{eq:asym-rates} follows. The smaller component rate is $R_m$
because $\C(G_m)\ge m+1$ and $\C(K_m)=m$.
\end{proof}

\end{document}

%% file: subsix_supplement.tex
\section{Proof of Theorem~\ref*{thm:sharper}(1): a counterexample below six}\label{sec:subsixproof}

This appendix gives the counterexample with both factor capacities
below six, proving \cref{thm:sharper}(1). We specify the graphs, codes
and exact matrix certificates that establish the required capacity
bounds, together with procedures for checking them.

The proof in this appendix combines analytical arguments with
computer-assisted verification of explicit graph data, codes and matrix
certificates. OpenAI Codex assisted with exploring candidate constructions
and certificates and with developing and reviewing the verification code.
The finite certificate checks use exact arithmetic, and the
product-capacity bound follows from the analytical spectral argument.

\subsection{Statement}

We use the conventions of Supplementary Apps.~\ref{sec:notation}--\ref{sec:spectrum}.
A length-$n$ code in $X$ is a clique in $X^{\OR n}$, and
$\C(X)=\sup_{n\ge1}\omega(X^{\OR n})^{1/n}$.

Set
\[
 b=\frac{32437757}{6250000}=5.19004112,\quad
 q=\frac{12991403}{2500000}=5.1965612,\quad
 L=\frac{5193957}{1000000}=5.193957.
\]
\begin{proposition}\label{subsix:thm:main}
For the graphs \(G,H\) defined in Subsec.~\ref{subsix:sec:graphs},
\[
 \C(G\times H)<L<\sqrt{27}\leq\min\{\C(G),\C(H)\},
\]
and, separately,
\[
 \C(G)\leq\frac{16}{3}<6,\qquad \C(H)\leq q<6.
\]
Consequently the Shannon OR-capacity Hedetniemi equality fails.
The same pair witnesses
\(\C(G\times H)<r<\min\{\C(G),\C(H)\}\) for every real
\(r\in[L,\sqrt{27})\).
\end{proposition}
This proves \cref{thm:sharper}(1).
The smaller factor capacity lies in $[\sqrt{27},q]$;
its exact value is unknown. The construction establishes this local
interval, with no claim that its lower endpoint is optimal or that
failure extends to every larger threshold.

\subsection{The graphs and their lower-bound codes}\label{subsix:sec:graphs}

\subsubsection{The 54-vertex graph}

Use the following ordered list of vectors in \(\F_3^5\), with entries
interpreted modulo three.
\begin{center}
\small
\begin{tabular}{r@{\quad}l r@{\quad}l r@{\quad}l}
\toprule
\(i\)&\(w_i\)&\(i\)&\(w_i\)&\(i\)&\(w_i\)\\
\midrule
0&(0,0,0,0,1) & 9&(0,1,2,0,0) & 18&(1,1,1,2,1)\\
1&(0,0,0,1,1) & 10&(0,1,2,2,1) & 19&(1,1,1,2,2)\\
2&(0,0,0,1,2) & 11&(0,1,2,2,2) & 20&(1,1,2,1,1)\\
3&(0,0,1,2,0) & 12&(1,0,0,0,0) & 21&(1,1,2,2,0)\\
4&(0,1,0,1,0) & 13&(1,0,1,0,0) & 22&(1,2,1,1,0)\\
5&(0,1,0,2,0) & 14&(1,0,1,1,2) & 23&(1,2,1,2,1)\\
6&(0,1,1,0,0) & 15&(1,0,2,0,0) & 24&(1,2,1,2,2)\\
7&(0,1,1,1,1) & 16&(1,0,2,2,2) & 25&(1,2,2,1,1)\\
8&(0,1,1,1,2) & 17&(1,1,1,1,0) & 26&(1,2,2,2,0)\\
\bottomrule
\end{tabular}
\end{center}
Let \(B\) have vertex set \(\{0,\ldots,26\}\) and adjacency
\[
 i\sim_Bj\quad\Longleftrightarrow\quad
 i\ne j\ \text{and}\ w_i\cdot w_j=0\text{ in }\F_3.
\]
Let \(D=\overline B\), the simple complement on the same vertex set.
Define
\[
 G=B\sqcup D.
\]
For explicit labels, the \(B\)-copy uses \(0,\ldots,26\) and the \(D\)-copy
uses \(27,\ldots,53\), with \(27+i\) representing vertex \(i\) of \(D\).
There are no edges between the copies.

The \(27\) words
\[
 \mathcal K_G=\{(i,27+i):0\leq i\leq26\}\subseteq V(G)^2
\]
form a code because, for distinct \(i,j\), exactly one of \(i\sim_Bj\) and
\(i\sim_Dj\) holds. Therefore
\begin{equation}\label{subsix:eq:Glower}
 \C(G)\geq\omega(G\OR G)^{1/2}\geq\sqrt{27}.
\end{equation}

\subsubsection{The 27-vertex graph}

Let \(V(H)=\mathbb Z_{29}\setminus\{1,17\}\), ordered as
\(0,2,3,\ldots,16,18,19,\ldots,28\). The following table defines the graph completely
by listing, in row \(v\), all its neighbours larger than \(v\).
Each listed pair is an undirected edge, and there are no loops.
The residue labels specify the vertices of \(H\); adjacency is defined
entirely by the table.
\begin{longtable}{r@{\quad}p{0.84\textwidth}}
\toprule
\(v\)&Larger neighbours in \(H\)\\
\midrule
\endfirsthead
\toprule
\(v\)&Larger neighbours in \(H\), continued\\
\midrule
\endhead
0&\(2,\,4,\,6,\,7,\,8,\,12,\,13,\,16,\,21,\,22,\,23,\,25,\,27\)\\
2&\(4,\,5,\,6,\,7,\,8,\,9,\,10,\,12,\,15,\,22,\,23\)\\
3&\(8,\,9,\,11,\,14,\,15,\,20,\,22,\,23,\,26,\,28\)\\
4&\(5,\,6,\,7,\,8,\,10,\,13,\,15,\,18,\,20,\,22,\,23,\,25,\,28\)\\
5&\(6,\,7,\,9,\,11,\,12,\,13,\,16,\,18,\,19,\,21,\,22,\,24\)\\
6&\(8,\,9,\,10,\,15,\,18,\,20,\,25,\,26,\,27\)\\
7&\(10,\,11,\,13,\,14,\,18,\,19,\,20,\,24,\,25,\,26,\,27,\,28\)\\
8&\(9,\,10,\,12,\,16,\,24\)\\
9&\(10,\,11,\,12,\,13,\,15,\,16,\,20,\,22,\,23,\,25,\,26\)\\
10&\(11,\,13,\,22,\,24,\,28\)\\
11&\(12,\,13,\,14,\,15,\,22,\,23,\,24,\,25\)\\
12&\(13,\,14,\,15,\,16,\,18,\,20,\,21,\,24,\,27,\,28\)\\
13&\(15,\,16,\,20,\,21,\,24\)\\
14&\(15,\,16,\,18,\,20,\,23,\,25,\,26,\,27,\,28\)\\
15&\(18,\,22,\,26,\,28\)\\
16&\(18,\,19,\,20,\,22,\,24,\,25\)\\
18&\(19,\,20,\,22,\,24,\,26\)\\
19&\(20,\,21,\,22,\,23,\,25,\,27,\,28\)\\
20&\(21,\,23,\,24,\,26,\,27\)\\
21&\(23,\,25,\,26,\,27\)\\
22&\(24,\,25,\,27,\,28\)\\
23&\(24,\,25,\,27\)\\
24&\(25,\,27\)\\
25&\(27,\,28\)\\
26&\(28\)\\
27&\(\varnothing\)\\
28&\(\varnothing\)\\
\bottomrule
\end{longtable}
This graph has \(176\) edges. The \(27\) words
\[
 \mathcal K_H=\{(i,12i\bmod29):i\in\mathbb Z_{29}\setminus\{1,17,28\}\}
 \cup\{(28,12)\}
\]
are all in \(V(H)^2\), since \(12i=1\) and \(12i=17\) modulo \(29\)
have the unique solutions \(i=17\) and \(i=28\), respectively.
Their first coordinates are distinct, and the table verifies
\[
 (a,b),(c,d)\in\mathcal K_H,\quad(a,b)\ne(c,d)
 \ \Longrightarrow\
 a\sim_Hc\ \text{or}\ b\sim_Hd.
\]
This is a finite check of \(351\) unordered pairs, also performed by
the exact verifier described in Subsec.~\ref{subsix:app:verification}.
Hence
\begin{equation}\label{subsix:eq:Hlower}
 \C(H)\geq\sqrt{27}.
\end{equation}

\subsection{The finite-field certificates}\label{subsix:sec:rank}

A block representation over $\F_3$ of a graph \(X\), with local block size \(d\)
and ambient dimension \(R\), consists of matrices
\(P_v\in\F_3^{d\times R}\), \(Q_v\in\F_3^{R\times d}\) such that
\[
 P_vQ_v=I_d,\qquad P_vQ_w=0\quad\text{whenever }v\sim_Xw.
\]
Both orientations of an edge are included. The resulting block matrix
$(P_vQ_w)_{v,w}$ need not be symmetric and has rank at most $R$.
By the definition of fractional Haemers' bound in
\suppref{sec:spectrum} and \cref{lem:haemers}, these factors certify
\[
 h_3(X):=\phif{\F_3}(X)\le R/d.
\]
Both $h_3$ and $\overline\vartheta(X)=\vartheta(\overline X)$ are OR
spectral points, as recalled in \cref{sec:spectral-argument} and
\cref{lem:haemers}.

\subsubsection{A rank-five scalar representation of B}

Every listed vector satisfies \(w_i\cdot w_i\ne0\) in \(\F_3\).
Define the row and column factors
\[
 \ell_i=(w_i\cdot w_i)^{-1}w_i^{\mathsf T},
 \qquad v_i=w_i,
\]
where the first expression is regarded as a \(1\times5\) row and
the second as a \(5\times1\) column. Then
\[
 \ell_i v_i=1,\qquad \ell_i v_j=0\quad(i\sim_Bj).
\]
Thus $h_3(B)\le5$.

\subsubsection{A rank-eleven, block-two representation of H}

The following eleven subsets are independent in \(H\), and every vertex
belongs to exactly two of them.
\begin{center}
\begin{tabular}{r@{\quad}l}
\toprule
\(a\)&Independent set \(I_a\)\\
\midrule
1&\{2,\,3,\,13,\,18,\,25\}\\
2&\{3,\,6,\,7,\,16,\,21\}\\
3&\{8,\,11,\,18,\,27,\,28\}\\
4&\{8,\,13,\,22,\,23,\,26\}\\
5&\{0,\,9,\,14,\,19,\,24\}\\
6&\{4,\,9,\,14,\,21,\,24\}\\
7&\{2,\,11,\,16,\,27,\,28\}\\
8&\{0,\,5,\,10,\,15,\,20\}\\
9&\{5,\,10,\,15,\,20,\,25\}\\
10&\{6,\,7,\,12,\,22,\,23\}\\
11&\{4,\,12,\,19,\,26\}\\
\bottomrule
\end{tabular}
\end{center}
Both claims can be checked directly from the adjacency table.
Let $e_1,\ldots,e_{11}$ be the standard coordinate vectors of $\F_3^{11}$.
For \(h\in V(H)\), let \(a(h)<b(h)\) be its two set indices and put
\[
 E_h=[\,e_{a(h)}\ \ e_{b(h)}\,]\in\F_3^{11\times2}.
\]
Then \(E_h^{\mathsf T}E_h=I_2\), and adjacent vertices have disjoint
index pairs, so \(E_h^{\mathsf T}E_{h'}=0\) on every \(H\)-edge.
Thus $h_3(H)\le11/2$; the same representation works over every field.

\subsection{Exact real Gram certificates}\label{subsix:sec:gram}

The following Gram-matrix criterion gives the theta bounds we need.
\begin{lemma}\label{subsix:lem:gram}
Suppose a real positive semidefinite matrix \(M\), indexed by \(V(X)\),
has diagonal entries \(t-1\) and entries \(-1\) on every edge of \(X\).
Then there are unit vectors \(u_v\) and a fixed unit vector \(c\),
called the handle, satisfying
\[
 u_v\perp u_w\quad(v\sim_Xw),\qquad
 \langle c,u_v\rangle=t^{-1/2}.
\]
Consequently $\overline\vartheta(X)\le t$ and $\C(X)\le t$.
\end{lemma}
\begin{proof}
Write \(M_{vw}=\langle z_v,z_w\rangle\), using a real Gram factorization.
Take \(u_v=(1,z_v)/\sqrt t\) and \(c=(1,0,\ldots,0)\).
The orthonormal-representation bound recalled in Supplementary
App.~\ref{sec:finiteproof} gives $\overline\vartheta(X)\le t$.
\end{proof}

Subsec.~\ref{subsix:app:matrices} supplies the following certificates.
\begin{center}
\begin{tabular}{c r r r c}
\toprule
Graph&Order&Scale $s_X$&Diagonal of $M_X$&Parameter $t_X$\\
\midrule
$B$&27&$24$&$104$&$16/3$\\
$D$&27&$100000000$&$419004112$&$b$\\
$H$&27&$100000000$&$419656120$&$q$\\
\bottomrule
\end{tabular}
\end{center}
Each $M_X$ is positive definite, has entry $-s_X$ on graph edges,
and has normalized diagonal $t_X-1$. The symmetry reductions below
verify positivity on blocks of size at most $6$ for $B,D$, and of sizes
$15$ and $12$ for $H$. Applying Lem.~\ref{subsix:lem:gram} to the
normalized matrices $M_X/s_X$ gives
\[
 \overline\vartheta(B)\le16/3,\qquad
 \overline\vartheta(D)\le b,\qquad
 \overline\vartheta(H)\le q.
\]
The maximum rule~\eqref{eq:max} therefore proves
\begin{equation}\label{subsix:eq:factoruppers}
 \C(G)\le\overline\vartheta(G)
 =\max\{\overline\vartheta(B),\overline\vartheta(D)\}\le16/3<6,
 \qquad \C(H)\le q<6.
\end{equation}

\subsection{The product bound}\label{subsix:sec:weighted}

The graph $G\times H$ is $(B\times H)\sqcup(D\times H)$.
Projecting to the appropriate factor in each component gives
\[
\begin{array}{c|cc}
 &h_3&\overline\vartheta\\ \hline
 B\times H&\le5&\le q\\
 D\times H&\le11/2&\le b .
\end{array}
\]
Apply~\eqref{eq:weighted-main} with these two spectral points. For every
$0\le\lambda\le1$,
\begin{equation}\label{subsix:eq:productupper}
 \C(G\times H)\le T_\lambda
 :=\max\{5^\lambda q^{1-\lambda},(11/2)^\lambda b^{1-\lambda}\}.
\end{equation}
This already controls arbitrary codes of every length, including words
whose coordinates lie in different components.

\subsection{Exact verification of the strict separation}

Choose the rational weight \(\lambda=13/1000\). The required strict
comparisons reduce to the rational inequalities
\begin{align}
 5^{13}\left(\frac{12991403}{2500000}\right)^{987}
 &<
 \left(\frac{5193957}{1000000}\right)^{1000},\label{subsix:eq:power1}\\
 \left(\frac{11}{2}\right)^{13}
 \left(\frac{32437757}{6250000}\right)^{987}
 &<
 \left(\frac{5193957}{1000000}\right)^{1000}.\label{subsix:eq:power2}
\end{align}
The inequalities~\eqref{subsix:eq:power1}--\eqref{subsix:eq:power2} are
checked by integer powers and cross multiplication; the exact
Python statements are given below. Taking positive \(1000\)-th roots
proves \(T_\lambda<L\). Also
\[
 5193957^2=26977189317849<27000000000000=27\cdot10^{12},
\]
so \(L<\sqrt{27}\). Combining \eqref{subsix:eq:Glower}, \eqref{subsix:eq:Hlower},
\eqref{subsix:eq:factoruppers}, and \eqref{subsix:eq:productupper} proves
Prop.~\ref{subsix:thm:main}.

To establish the sub-six spectral claim in \cref{sec:spectral-consequences}, choose the
single interpolating point $\psi\in\spec$ supplied by
\eqref{eq:interpolate} for $h_3,\overline\vartheta$ with weights
$13/1000,987/1000$. The component bounds above and the strict
comparisons~\eqref{subsix:eq:power1}--\eqref{subsix:eq:power2} give
$\psi(G\times H)\le T_{13/1000}<L$.
Every spectral point bounds capacity from above, so
$\psi(G),\psi(H)\ge\sqrt{27}$.
To bound their values from above, the identity fitting matrix gives
$h_3(D)\le27$, and the maximum rule gives $h_3(G)\le27$.
The existing certificates also give
$\overline\vartheta(G)\le16/3$, $h_3(H)\le11/2$ and
$\overline\vartheta(H)\le q<11/2$.
Thus the same interpolant satisfies
\[
 \psi(G)\le27^{13/1000}(16/3)^{987/1000}
 \le\frac{13}{1000}\,27+\frac{987}{1000}\,\frac{16}{3}
 =\frac{1123}{200}<6,
\]
by the weighted arithmetic--geometric mean inequality, while
\[
 \psi(H)\le(11/2)^{13/1000}q^{987/1000}<11/2<6.
\]
Thus one spectral point has product value below $\sqrt{27}$ and
both factor values in $[\sqrt{27},6)$, as asserted in
\cref{sec:spectral-consequences}.

\subsection{Compact integer matrix certificates}\label{subsix:app:matrices}

The scales and diagonals are those in Subsec.~\ref{subsix:sec:gram}.
Every edge entry is minus the scale. Only nonedge entries remain to be
specified. In each table, \texttt{i: j=value} gives an entry at an unordered
pair $\{i,j\}$; extend it to the whole orbit under the stated permutations,
and set $M_{ji}=M_{ij}$. For the group $\Gamma$ generated by those
permutations, this pair orbit is
$\{\{g(i),g(j)\}:g\in\Gamma\}$.
Each nonedge orbit occurs exactly once.

\subsubsection{The two source matrices}

Both $B$ and $D$ are preserved by the permutations
\begin{align*}
 a&=(4\,5)(6\,9)(7\,11)(8\,10)(17\,22)(18\,23)(19\,24)(20\,25)(21\,26)\\
 b_0&=(0\,17)(1\,19)(4\,23)(6\,25)(7\,12)(8\,26)(11\,16)(13\,24)(14\,22)\\
 c&=(0\,16)(2\,15)(4\,25)(5\,20)(6\,23)(7\,22)(9\,18)(11\,17)(12\,14).
\end{align*}
They generate a group $\Gamma$ of order $16$. The following $23$ and $19$
entries determine $M_B$ and $M_D$, respectively.

\noindent\textbf{Source certificate $M_B$:}
\begin{Verbatim}[fontsize=\scriptsize,baselinestretch=1.03]
 0: 1=24 2=6 7=15 8=22 14=-9 16=12 18=18 19=27 20=-2
 1: 3=-31 4=4 19=11 21=5
 2: 3=-4 4=20 8=27
 3: 4=35
 4: 6=2 8=-20 9=19 23=23 26=34
 8: 10=6
\end{Verbatim}
\noindent\textbf{Complement certificate $M_D$:}
\begin{Verbatim}[fontsize=\scriptsize,baselinestretch=1.03]
 0: 3=92395858 4=145133660 6=-30688165 12=49988928 13=182327057 15=132541340
    17=84507591 21=76315484
 1: 2=3087050 6=138042021 8=109278319 13=-87574439
 2: 6=127958661 15=-3681718
 3: 8=40039826
 4: 5=47798171 10=83490663 25=123407703
 8: 26=-16181236
\end{Verbatim}

Here is an elementary reduction of their positivity checks. Write
$A,B_0,C$ for the permutation matrices of $a,b_0,c$, and put $Z=(AB_0)^2$.
These are symmetric involutions; $C,Z$ commute with $A,B_0$, and
$AB_0=ZB_0A$. Thus $A$ and $B_0$ commute on the $Z=+I$ eigenspace
and anticommute on the $Z=-I$ eigenspace.

Let $\varepsilon,\delta,\alpha,\beta\in\{+1,-1\}$ label the eigenvalues
of $Z,C,A,B_0$, respectively. When $\varepsilon=+1$, the simultaneous
eigenspaces of these four commuting symmetric matrices form an orthogonal
decomposition. When $\varepsilon=-1$, decompose instead by the
eigenvalues $\delta,\alpha$ of $C,A$. The orthogonal map $B_0$ carries
each such space onto the space with signs $\delta,-\alpha$.
A $\Gamma$-invariant symmetric matrix commutes with every generator,
preserves these spaces, and has orthogonally equivalent restrictions
on each paired set of spaces. It therefore suffices to check one member
of each pair, explaining the multiplicity $2$ below.

To obtain explicit integer bases, apply
\[
 (I+\beta B_0)(I+\alpha A)(I+\delta C)(I+\varepsilon Z)
\]
to $e_0,e_1,\ldots,e_{26}$ when $\varepsilon=+1$, and omit the first
factor when $\varepsilon=-1$. Divide each nonzero column by the gcd of its
entries, then retain the first linearly independent columns. Denote the
resulting basis matrix by $P$. The nonzero spaces are listed below;
when $\varepsilon=-1$, it suffices to take $\alpha=+1$, with multiplicity
$2$. The dimensions, counted with multiplicity, sum to $27$.

For each certificate, $P^{\mathsf T}M_XP$ represents the restricted
quadratic form in the chosen basis, which need not be orthonormal.
Its positive definiteness is therefore equivalent to that of the
restriction. Divide this matrix by the positive gcd of its entries.
The last column records the leading principal minors, the determinants
of the upper-left square submatrices, for the resulting $M_B$ blocks;
all are positive. Applying the exact
elimination below to the $M_D$ blocks likewise gives only positive minors.
\begin{center}
\small
\begin{tabular}{ccccrrl}
\toprule
$\varepsilon$&$\delta$&$\alpha$&$\beta$&Order&Multiplicity&Minors for $M_B$\\
\midrule
$+$&$+$&$+$&$+$&6&1&$130,1596,5376,2192,96,112$\\
$+$&$+$&$+$&$-$&2&1&$101,13$\\
$+$&$+$&$-$&$+$&2&1&$115,66$\\
$+$&$-$&$+$&$+$&3&1&$29,28,20$\\
$+$&$-$&$+$&$-$&1&1&$1$\\
$+$&$-$&$-$&$+$&1&1&$1$\\
$-$&$+$&$+$&--&4&2&$149,7232,10816,3786$\\
$-$&$-$&$+$&--&2&2&$77,94$\\
\bottomrule
\end{tabular}
\end{center}
Thus both source certificates are positive definite, using blocks of
order at most six.

\Needspace{8\baselineskip}
\subsubsection{The partner matrix}

On $H$, let $\sigma$ fix $0,12,28$ and send every other vertex
$v$ to $29-v$. It preserves adjacency. The next $91$ entries, extended
under $\sigma$, specify $M_H$.
\begin{Verbatim}[fontsize=\scriptsize,baselinestretch=1.03]
 0: 3=-92182906 5=112385240 9=105366413 10=260048876 11=-64530825
    14=103723987 28=24358075
 2: 3=66703246 11=207795723 13=104538630 14=-136454182 16=177512886
    18=115342751 19=-34383691 20=-100268450 21=76229467 24=-95624150
    25=120690256 26=152316144 27=18986268 28=103951980
 3: 4=-35794660 5=-98567795 6=149505933 7=190124261 10=-22994873 12=66710199
    13=124899718 16=61341650 18=120482232 19=-19609272 21=-23529434
    24=-98379162 25=42147361
 4: 9=272941041 11=23305245 12=-70035550 14=241110784 16=61213470
    19=135393685 21=30684152 24=262861878
 5: 8=35191781 10=321011184 14=-104169070 15=375111138 20=418775902
    23=-98757666 28=-47386121
 6: 7=381208558 11=-42891154 12=337627573 13=-146194020 14=-78782821
    16=19526989 19=-86283590 21=90671780 22=187815754 23=202130330
    28=-36203985
 7: 8=-59663804 9=-100785398 12=291996527 15=-44507231 16=55398098
    21=117706337 22=165147825
 8: 11=-56920029 13=259229426 14=-47902526 15=57463096 18=183702360
    19=-222235396 20=41480973 21=10556935 28=92833206
 9: 14=379995036 18=-98046773 19=316558361 28=-53434665
10: 12=-37639343 14=-98004443 15=294664038 16=-99343891 18=-90692258
    19=60606640
11: 16=300836730 18=35163766 28=237971338
13: 14=-57898358 28=23099334
\end{Verbatim}

For the positivity check, order the vertices as $F,P,\sigma(P)$, where
\[
 F=(0,12,28),\qquad P=(2,3,4,5,6,7,8,9,10,11,13,14).
\]
The invariant matrix has the form
\[
 M_H=\begin{pmatrix}E&F_0&F_0\\
 F_0^{\mathsf T}&J&K\\F_0^{\mathsf T}&K&J\end{pmatrix}.
\]
Its quadratic form splits into symmetric and antisymmetric coordinates.
Consequently it is positive definite if and only if
\[
 S=\begin{pmatrix}2E&2F_0\\2F_0^{\mathsf T}&J+K\end{pmatrix}
 \quad\text{and}\quad T=J-K
\]
are positive definite. Indeed, for coordinates $(x,y,z)$ of dimensions
$3,12,12$, respectively,
\[
 \begin{pmatrix}x\\y\\z\end{pmatrix}^{\mathsf T}
 M_H\begin{pmatrix}x\\y\\z\end{pmatrix}
 =\frac12\begin{pmatrix}x\\y+z\end{pmatrix}^{\mathsf T}
 S\begin{pmatrix}x\\y+z\end{pmatrix}
 +\frac12(y-z)^{\mathsf T}T(y-z).
\]
The change of variables $(x,y,z)\mapsto(x,y+z,y-z)$ is invertible,
which proves the equivalence. Exact elimination gives $15$ and $12$ positive
leading principal minors, respectively, proving that $M_H$ is positive
definite.

\Needspace{10\baselineskip}
\subsection{Exact verification and reproducibility}\label{subsix:app:verification}

All entries are specified above. The permutations, orbit coverage,
graph edges and block dimensions are finite checks. For positive
definiteness, we use Bareiss's fraction-free elimination~\cite{Bareiss}.
Sylvester's determinant identity shows that, at stage $k$,
\texttt{pivot} is the leading principal minor of order $k+1$ and all
divisions are exact. Positivity of these minors proves positive
definiteness by Sylvester's criterion.
\begin{Verbatim}[fontsize=\footnotesize,baselinestretch=.96]
def positive_definite(M):
    n = len(M)
    assert all(len(row) == n for row in M)
    assert all(M[i][j] == M[j][i]
               for i in range(n) for j in range(n))
    T = [row[:] for row in M]
    previous = 1
    for k in range(n):
        pivot = T[k][k]
        assert pivot > 0
        for i in range(k + 1, n):
            for j in range(i, n):
                num = pivot*T[i][j] - T[i][k]*T[k][j]
                assert num % previous == 0
                T[i][j] = T[j][i] = num // previous
        previous = pivot
    return n
\end{Verbatim}
The following code checks the strict scalar comparisons.
\begin{Verbatim}[fontsize=\footnotesize,baselinestretch=.96]
from fractions import Fraction as F
b = F(32437757, 6250000)
q = F(12991403, 2500000)
L = F(5193957, 1000000)
assert F(5)**13 * q**987 < L**1000
assert F(11, 2)**13 * b**987 < L**1000
assert L*L < 27 and max(F(16, 3), b, q) < 6
\end{Verbatim}
The verification programs and machine-readable data are available in
the accompanying repository~\cite{ZeissData}.
The program \nolinkurl{verification/check_subsix.py} verifies all finite
inputs and the symmetry reductions. The independent
\nolinkurl{verification/audit_subsix_from_tex.py} reconstructs the printed
data and checks the full matrices by rational $LDL^{\mathsf T}$
factorization. Run either with Python 3.11 or later and assertions enabled:
do not pass \texttt{-O} or set \texttt{PYTHONOPTIMIZE}. Only the standard
library is required. Both checkers use exact arithmetic, without floating-point
operations or a numerical solver. They verify the matrix certificates
used by the analytical interpolation argument.

\paragraph{AI assistance and verification workflow.}
OpenAI Codex assisted with the development and review of the verification
programs and with exploration of candidate constructions and certificates.
The archived inputs are the explicit graph definitions, finite codes and
integer or rational certificate entries printed in these appendices.
The programs check these fixed inputs; they do not rely on an AI service,
a numerical solver or an AI-generated assessment of correctness at run time.
In particular, the source-derived checker reconstructs the printed matrices
independently of the modular checker. Integer and rational computations
establish the finite certificate inequalities. Decimal evaluation in the
uniform-threshold audit is diagnostic only. The spectral interpolation
and asymptotic arguments are proved analytically in the appendices and
are not certified by these finite programs.

%% file: threshold_supplement.tex
\Needspace{16\baselineskip}
\section{Proof of Theorem~\ref*{thm:sharper}(2): coverage of real thresholds}\label{sec:thresholdproof}

This appendix proves the uniform threshold statement in
\cref{thm:sharper}(2). We construct graph pairs covering overlapping
ranges of separating thresholds, then extend this coverage to every
real threshold from $6.0101$ by joining complete graphs.

This appendix combines an analytical construction with computer-assisted
checks of finite graph data and threshold inequalities. OpenAI Codex
assisted with exploring the constructions and with developing and
reviewing the verification code. The finite threshold comparisons and
interval overlaps are checked using exact integer and rational
arithmetic; the extension to all thresholds in the claimed range is
proved analytically.

\begin{theorem}[Failure at every real threshold from $6.0101$]
\label{thm:real-thresholds}
For every real $t\ge6.0101$, there are finite simple graphs $G,H$ such that
\[
 \boxed{\C(G\times H)<t<\min\{\C(G),\C(H)\}.}
\]
In particular, $f_C(t)<t$ throughout this range.
\end{theorem}

The proof gives the stronger range $t>T_0$, where
$T_0=6.0100886404\ldots$ is defined exactly
in~\eqref{eq:uniform-optimized}. We first optimize a counterexample
near six, then use eight witnesses to cover a full unit interval and
extend that cover by joining complete graphs. The optimization is
analytical, and all comparisons for the stated decimal cutoff use
exact integer arithmetic.

We use the weighted product bound~\eqref{eq:weighted-main} from
\suppref{sec:interpolation}.
Throughout this section, put
\[
 A=U_6,\qquad D=\overline A,\qquad
 h(X)=\phif{\F_2}(X),\qquad q(X)=\vartheta(\overline X).
\]
Both $h$ and $q$ are OR spectral points. The constructions in
\cref{lem:binary,lem:clebsch} give
\[
 h(A)=6,\qquad q(D)\le4,\qquad
 \C(Q)\ge16^{1/3},\qquad q(Q)\le8/3.
\]
We first provide the remaining small-graph estimates.

\subsection{Fractional colourings and a 37-vertex partner}

We use the fractional-colouring definitions in
\hyperref[def:fractional-colouring]{Supplementary App.~\ref*{sec:spectrum}}
and the inequality $h(X)\le\chi_f(X)$ proved in
\eqref{eq:haemers-colouring} by the coordinate-vector fitting construction.

In particular,
\begin{equation}\label{eq:threshold-colouring-D}
 h(D)\le\chi_f(D)\le\frac{16}{3}.
\end{equation}
Indeed, $A$ is vertex-transitive. Coordinate permutations and the maps
$T_v(x)=x+(x\cdot v)v$ for even-weight $v$ preserve dot products
and odd weight; moreover, $T_v^2$ is the identity. Taking
$v=e_1+e_2+e_3+e_4$ or $v=\mathbf1$ sends $e_1$ to a vector of
weight three or five. Permutations then reach every odd-weight vector.
Thus $D$ is also vertex-transitive. Averaging its independent six-set
$\{e_1,\ldots,e_6\}$ over the automorphism group covers every vertex
equally. Normalizing this coverage to one gives a fractional colouring
of total weight $32/6$, as explained in \suppref{sec:spectrum}, and proves
\eqref{eq:threshold-colouring-D}.

\begin{lemma}[A self-complementary circulant]\label{lem:circulant37}
Let $R$ be the graph on $\mathbb Z/37\mathbb Z$ in which two vertices
are adjacent when their undirected cyclic distance belongs to
\[
 S=\{6,7,8,9,10,12,13,15,18\}.
\]
Then
\[
 \C(R)=q(R)=\sqrt{37},\qquad h(R)\le\chi_f(R)=\frac{37}{6}.
\]
\end{lemma}
\begin{proof}
Multiplication by six exchanges the two members of each distance pair
\[
 (1,6),(2,12),(3,18),(4,13),(5,7),
 (8,11),(9,17),(10,14),(15,16).
\]
Since $S$ chooses one member of every pair, this map is an isomorphism
$R\cong\overline R$. Consequently the 37 words $(i,6i)$ form a clique
in $R\OR R$, proving $\C(R)\ge\sqrt{37}$. The graph is
vertex-transitive, so Lov\'asz's identity
$\vartheta(R)\vartheta(\overline R)=37$~\cite{Lovasz}, together with
self-complementarity, gives the matching theta upper bound.

The set $B=\{0,6,12,18,24,30\}$ is a six-clique because its nonzero
undirected differences lie in $\{6,12,18,13,7\}$. Thus $6B$ is an
independent six-set. Its 37 translates cover every vertex six times,
giving $\chi_f(R)\le37/6$. Conversely, every independent set has
size at most $\vartheta(R)=\sqrt{37}<7$, hence at most six. Summing
the covering constraints of any fractional colouring gives
$37\le6\chi_f(R)$. Thus $\chi_f(R)=37/6$, proving the claim.
\end{proof}

\subsection{The initial pair near six}

Put $G=A\sqcup(D\join K_1)$. For any $\varphi\in\spec$, write
$x=\varphi(A)$ and $y=\varphi(D)$. The 32 diagonal words in
$A\OR D$ form a clique, so $xy\ge32$. The join and maximum rules give
$\varphi(G)=\max\{x,y+1\}$. If this maximum is $r$, then
$32\le xy\le r(r-1)$, and hence
\begin{equation}\label{eq:near-six-factors}
 \C(G)\ge\frac{1+\sqrt{129}}2>\sqrt{37}=\C(R).
\end{equation}
The projections give the spectral upper bounds
\[
\begin{array}{c|cc}
 &h&q\\ \hline
 A\times R&6&\sqrt{37}\\
 (D\join K_1)\times R&37/6&5.
\end{array}
\]
The geometric bound in~\eqref{eq:weighted-main} gives
\[
 \C(G\times R)\le
 \min_{0\le\lambda\le1}
 \max\{6^\lambda(\sqrt{37})^{1-\lambda},
          (37/6)^\lambda5^{1-\lambda}\}.
\]
Put $d=\log(\sqrt{37}/6)>0$ and $\ell=\log(6/5)>0$,
using natural logarithms.
After subtracting $\log6$, the logarithms of the two branches are
$(1-\lambda)d$ and $2\lambda d-(1-\lambda)\ell$.
The first decreases and the second increases with $\lambda$.
They agree at $\lambda_*=(\ell+d)/(\ell+3d)\in(0,1)$,
which therefore minimizes their maximum. Consequently
\begin{equation}\label{eq:uniform-optimized}
 \C(G\times R)\le T_0
 :=6\exp\!\left(\frac{2d^2}{\ell+3d}\right)
 =6.0100886404\ldots .
\end{equation}
Together with~\eqref{eq:near-six-factors}, this proves failure at every
threshold $T_0<t<\sqrt{37}$.

To prove $T_0<6.0101$ exactly, without using the decimal approximation
in~\eqref{eq:uniform-optimized}, use $\lambda=93/106$ in the geometric
bound. Direct integer exponentiation gives
\begin{equation}\label{eq:uniform-integer-check}
\begin{aligned}
 6^{186}37^{13}10000^{212}&<60101^{212},\\
 37^{93}5^{13}10000^{106}&<6^{93}60101^{106}.
\end{aligned}
\end{equation}
Taking the positive $212$th and $106$th roots, respectively, shows that
both branches are strictly below $60101/10000=6.0101$.
The supplementary script \nolinkurl{verification/check_uniform_cutoff.py}
in the accompanying repository~\cite{ZeissData} checks exactly these
two integer comparisons. The same repository includes
\nolinkurl{verification/check_exact.py} and
\nolinkurl{verification/audit_uniform_independent.py}, which check the
finite graph inputs and interval comparisons used below.

For the interval cover below, we retain a simpler arithmetic bound.
Use \eqref{eq:weighted-main} with $\lambda=13/15$ and
$\sqrt{37}<73/12$. The two arithmetic means agree, giving
\[
 \C(G\times R)
 \le\frac{13}{15}\,6+\frac{2}{15}\frac{73}{12}
 =\frac{13}{15}\frac{37}{6}+\frac{2}{15}\,5
 =\frac{541}{90}.
\]
Thus this pair also supplies the interval $(541/90,\sqrt{37})$.
The next three subsections cover every $t>541/90$; combining that
cover with the optimized initial interval proves the theorem.

\subsection{The cycle inputs}

Write $C_n$ for the cycle on $\mathbb Z/n\mathbb Z$, with two vertices
adjacent when their difference is $1$ or $-1$ modulo $n$ ($n\ge3$).
For the pentagon $C_5$, the standard theta calculation and its five-word
square code give
\[
 \C(C_5)=q(C_5)=\sqrt5,\qquad h(C_5)\le\chi_f(C_5)=5/2.
\]
Explicitly, the words $(i,2i)$, for $i\in\mathbb Z/5\mathbb Z$, form
a clique in $C_5\OR C_5$. The upper bound is Lov\'asz's pentagon
calculation~\cite{Lovasz}; the five translates of an independent pair
give the fractional colouring.

Write $Z=\overline{C_9}$. We need only
\[
 \C(Z)\ge81^{1/3},\qquad h(Z)\le9/2,\qquad q(Z)<22/5.
\]
For the first inequality, take the 81 words
\[
 (x,y,2x+4y),\qquad x,y\in\mathbb Z/9\mathbb Z.
\]
If two distinct words were nonadjacent in $Z^{\OR3}$, their coordinate
differences would have representatives $a,b,c\in\{-1,0,1\}$ with
$c-2a-4b\equiv0\pmod9$. Its absolute value is at most seven, so this
is an integer equality. Since $4>2+1$ and $2>1$, it forces
$a=b=c=0$, a contradiction. The nine translates of an edge of $C_9$
are independent pairs of $Z$, proving the colouring bound $9/2$.
Finally, the odd-cycle theta formula~\cite{Lovasz} gives
\[
 q(Z)=\vartheta(C_9)
 =\frac{9\cos(\pi/9)}{1+\cos(\pi/9)}<\frac{22}{5}.
\]
For a rational verification of the strict inequality, put $x=\pi/9$.
Since $1/3<x<1$, the alternating Taylor estimate gives
\[
 \cos x\le1-\frac{x^2}{2}+\frac{x^4}{24}
 \le1-\frac{11x^2}{24}<\frac{205}{216},\qquad
 \frac{9(205/216)}{1+205/216}=\frac{1845}{421}<\frac{22}{5}.
\]

\subsection{Balancing the bounds and adding cliques}

For an auxiliary graph $J$, define
\[
 G_J=A\sqcup(D\join J),\qquad
 \rho_b=\frac{b+\sqrt{b^2+128}}2.
\]
If $\C(J)\ge b>0$, then \cref{lem:joined-complement} gives
$\C(G_J)\ge\rho_b$. The following lemma chooses the interpolation
weight once and for all from three projection bounds.

\begin{lemma}[Balancing and adding cliques]\label{lem:threshold-padding}
Let $\C(J)\ge b>0$, and let $H$ be a partner graph. Suppose
$a>6$ and $s>d$ satisfy
\[
 s\ge q(H),\qquad
 a\ge\min\{h(D\join J),h(H)\},\qquad
 d\ge q(D\join J).
\]
Put
\begin{equation}\label{eq:threshold-balanced}
 U=\frac{as-6d}{a-6+s-d}.
\end{equation}
For every integer $j\ge0$, there are graphs $G^{(j)},H^{(j)}$ with
\[
 \C(G^{(j)}\times H^{(j)})\le U+j,\qquad
 \C(G^{(j)})\ge\rho_b+j,\qquad
 \C(H^{(j)})=\C(H)+j.
\]
\end{lemma}
\begin{proof}
Take
\[
 G^{(j)}=(A\join K_j)\sqcup(D\join J\join K_j),\qquad
 H^{(j)}=H\join K_j,
\]
omitting $K_0$. The categorical product distributes over disjoint
union. Projection to the factors and additivity under join give the
following bounds on its two components.
\[
\begin{array}{c|cc}
 &h&q\\ \hline
 (A\join K_j)\times H^{(j)}&6+j&s+j\\
 (D\join J\join K_j)\times H^{(j)}&a+j&d+j.
\end{array}
\]
Choose $\lambda=(s-d)/(a-6+s-d)\in(0,1)$. By construction,
\[
 6\lambda+s(1-\lambda)=a\lambda+d(1-\lambda)=U.
\]
Apply the spectral interpolation inequality~\eqref{eq:interpolate}
to $h,q$ with these weights. The resulting single spectral point,
the maximum rule, and the weighted
arithmetic--geometric mean bound the product capacity by $U+j$.
Finally, every spectral point takes value $\varphi(G_J)+j$ on
$G^{(j)}$ and value $\varphi(H)+j$ on $H^{(j)}$.
Taking spectral minima gives the two factor claims.
\end{proof}

\subsection{A single coverage table}

In \cref{tab:threshold-cover}, $kC_5$ means the join of $k$ copies of $C_5$,
and every terminating decimal denotes an exact rational number.
The columns $s,a,d$ satisfy the preceding lemma, and $U$ is
given by~\eqref{eq:threshold-balanced}. Each row gives
\begin{equation}\label{eq:threshold-row}
 \C(G_J\times H)\le U<L<\min\{\C(G_J),\C(H)\}.
\end{equation}
Thus that row works for every threshold in the open interval $(U,L)$.

\begin{table}[htbp]
\centering
\small
\renewcommand{\arraystretch}{1.18}
\begin{tabular}{@{}llccccc@{}}
\toprule
$J$ & $H$ & $s$ & $a$ & $d$ & $U$ & $L$\\
\midrule
$K_1$ & $R$ & $73/12$ & $37/6$ & $5$ & $541/90$ & $6.08$\\
$K_1$ & $K_4\join C_5$ & $156/25$ & $19/3$ & $5$ & $357/59$ & $6.17$\\
$K_2$ & $K_4\join C_5$ & $156/25$ & $13/2$ & $6$ & $228/37$ & $6.23$\\
$K_2$ & $K_2\join Z$ & $32/5$ & $13/2$ & $6$ & $56/9$ & $6.325$\\
$K_2$ & $K_2\join2C_5$ & $162/25$ & $7$ & $6$ & $234/37$ & $6.47$\\
$K_2$ & $3C_5$ & $168/25$ & $22/3$ & $6$ & $498/77$ & $6.7$\\
$C_5$ & $K_7$ & $7$ & $7$ & $156/25$ & $289/44$ & $6.88$\\
$Q$ & $K_1\join R$ & $85/12$ & $43/6$ & $20/3$ & $775/114$ & $7.02$\\
\bottomrule
\end{tabular}
\caption{One formula certifies every interval $(U,L)$. The first factor
is $G_J=A\sqcup(D\join J)$; $s,a,d$ are spectral projection bounds.}
\label{tab:threshold-cover}
\end{table}

Here $s$ bounds $q(H)$ and $d$ bounds $q(D)+q(J)$.
For $a$, use $h(D)+h(J)$ in rows 2 and 6, using
\eqref{eq:threshold-colouring-D}, and $h(H)$ in all other rows.
All entries follow from the preceding graph bounds, additivity under
join, and the elementary enclosures
\[
 2.235<\sqrt5<2.24,\qquad
 2.5<16^{1/3},\qquad 4.325<81^{1/3},
\]
and $6.08<\sqrt{37}<73/12$. Each follows by squaring or cubing.

For the first factor, $L<\rho_b$ follows from $L>b$ and
$L(L-b)<32$. For each choice of $J$, it suffices to check the largest
$L$. Using the respective lower bounds $b=1,2,2.235,2.5$, we obtain
\[
\begin{aligned}
 6.17(6.17-1)&=31.8989<32,&
 6.7(6.7-2)&=31.49<32,\\
 6.88(6.88-2.235)&=31.9576<32,&
 7.02(7.02-2.5)&=31.7304<32.
\end{aligned}
\]
For the partners, the capacity lower bounds are, in row order,
\[
 \sqrt{37},\ 4+\sqrt5,\ 4+\sqrt5,\ 2+81^{1/3},\
 2+2\sqrt5,\ 3\sqrt5,\ 7,\ 1+\sqrt{37}.
\]
The same rational enclosures place each strictly above its listed $L$.
This proves~\eqref{eq:threshold-row} for every row.

\FloatBarrier
\begin{proof}[Completion of the proof of \cref{thm:real-thresholds}]
Write $(U_i,L_i)$ for the intervals in
\cref{tab:threshold-cover}. Direct cross-multiplication gives
$U_i<L_i$ and $U_{i+1}<L_i$; the smallest consecutive overlap
is $6.325-234/37=1/1480>0$. Their union therefore contains
$(541/90,7.02)$, an interval of length greater than one.

For any $t>c:=541/90$, choose $j=\lceil t-c\rceil-1\ge0$.
Then $c<t-j\le c+1<7.02$, so some row has $U_i<t-j<L_i$.
\Cref{lem:threshold-padding} gives
\[
 \C(G^{(j)}\times H^{(j)})\le U_i+j<t<L_i+j
 <\min\{\C(G^{(j)}),\C(H^{(j)})\}.
\]
Thus the same eight witnesses, extended by joining complete graphs,
cover every $t>541/90$.

The optimized initial pair covers $(T_0,\sqrt{37})$, and
\eqref{eq:uniform-integer-check} gives
\[
 T_0<6.0101<541/90<\sqrt{37}.
\]
The two ranges overlap, so together they cover every $t>T_0$.
In particular they cover every $t\ge6.0101$, including the stated
endpoint. This proves the uniform clause of \cref{thm:sharper}.
\end{proof}

For the real-threshold spectral claim in \cref{sec:spectral-consequences}, fix
$t\ge6.0101$ and choose a pair $G_t,H_t$ from the construction above.
By~\eqref{eq:duality}, there is a spectral point $\psi_t\in\spec$
attaining the capacity of $G_t\times H_t$. Since every spectral point
bounds capacity from above, we obtain
\[
 \psi_t(G_t\times H_t)=\C(G_t\times H_t)<t
 <\min\{\C(G_t),\C(H_t)\}
 \le\min\{\psi_t(G_t),\psi_t(H_t)\}.
\]
This is Simonyi's implication~\cite[Thm.~8]{Simonyi} applied to each
pair. The choice of $\psi_t$ may depend on $t$.

\subsection{Scope of the sharper thresholds}

The positive cases $1$ and $2$ recalled in \suppref{sec:spectrum},
together with \cref{thm:finite}, show that the smallest integer value of
the smaller factor capacity admitting a counterexample lies between
$3$ and $7$. The unresolved exact integer minima are $3,\ldots,6$. The pair in
\cref{sec:subsixproof} has smaller factor capacity in
$[\sqrt{27},5.1965612]$; its exact value is unknown.
The earlier pair in~\eqref{eq:near-six-factors} has exact smaller
factor capacity $\sqrt{37}$.
By \cref{subsix:thm:main}, the sub-six pair gives failure at every
threshold $5.193957\le t<\sqrt{27}$, whereas this section covers
every $t>T_0$. A cover joining the lower interval to this upper ray
is still missing. In particular, the strict-threshold formulation
does not include $t=\sqrt{27}$ merely from the lower bound
$\min\{\C(G),\C(H)\}\ge\sqrt{27}$, and the upper-ray proof does not
include $t=T_0$. This does not assert that each individual threshold
between these endpoints is unresolved.

Neither result classifies all possible exact noninteger minimum
capacities; no optimality is claimed.
Polak's circular-graph constructions provide, after complementation,
graphs whose OR-capacities approach any integer $r\ge3$ from
below~\cite[Chap.~9, Thms.~9.3.1 and~9.3.3]{PolakThesis}.
These constructions do not determine whether the product identity holds
at those capacities.